\documentclass[12pt]{amsart}

\usepackage{amssymb,latexsym}
\usepackage{verbatim,euscript,array}
\usepackage{fullpage,color}
\usepackage{tikz}
\usetikzlibrary{arrows,shapes,trees}

\usepackage[colorlinks=true,linkcolor=blue,urlcolor=my_color,citecolor=magenta]{hyperref}

\definecolor{my_color}{rgb}{0,0.5,0.5}
\definecolor{MIXT}{rgb}{0.4,0.3,0.6}
\definecolor{mixt}{rgb}{0.5,0.3,0.2}
\definecolor{sin}{rgb}{0,0.5,0.5}
\definecolor{darkblue}{rgb}{0,0.1,0.8}
\definecolor{redi}{rgb}{0.5,0,0.4}

\input {cyracc.def}
\numberwithin{equation}{section}

\font\tencyr=wncyr10 %scaled \magstephalf
\def\rus{\tencyr\cyracc}

\newtheorem{thm}{Theorem}[section]
\newtheorem{lm}[thm]{Lemma}%[chapter]
\newtheorem{cl}[thm]{Corollary}%[chapter]
\newtheorem{prop}[thm]{Proposition}%[chapter]

\theoremstyle{remark}
\newtheorem{rmk}[thm]{Remark}

\theoremstyle{definition}
\newtheorem{ex}[thm]{Example}
\newtheorem{df}{Definition}
\newtheorem*{rema}{Remark}

\newcommand{\eus}{\EuScript}
\newcommand {\be}{{\mathfrak b}}

\newcommand {\g}{{\mathfrak g}}

\newcommand {\te}{{\mathfrak t}}
\newcommand {\ut}{{\mathfrak u}}

\newcommand {\gN}{{\eus N}}
\newcommand {\gH}{{\eus H}}

\newcommand {\esi}{\varepsilon}
\newcommand {\ap}{\alpha}

\newcommand {\lb}{\lambda}

\newcommand {\HW}{\widehat W}
\newcommand {\HV}{\widehat V}
\newcommand {\HP}{\widehat\Pi}
\newcommand {\HD}{\widehat\Delta}

\newcommand {\ct}{{\mathcal T}}

\newcommand {\BR}{{\mathbb R}}
\newcommand {\BN}{{\mathbb N}}

\newcommand {\kl}{{\mathsf {cl}}}

\newcommand {\hot}{{\mathsf{ht}}}

\newcommand {\rt}{{\mathsf{rt}}}

\newcommand {\supp}{{\mathsf{supp}}}

\newcommand {\GR}[2]{{\textrm{{\sf\bfseries #1}}}_{#2}}

\newcommand {\Ab}{\mathfrak{Ab}}
\newcommand {\Abo}{\mathfrak{Ab}^o}%{\overset{o}{\mathfrak{Ab}}}
\newcommand {\AD}{\mathfrak{Ad}}

\newcommand {\thi}{{\lfloor\theta/2\rfloor}}
\newcommand {\tthe}{{\lceil \theta/2\rceil}}
\newcommand {\bap}{\boldsymbol{\hat\ap}}

\newcommand {\un}{\underline}
\newcommand {\beq}{\begin{equation}}
\newcommand {\eeq}{\end{equation}}

\newcommand{\curge}{\succcurlyeq}
\newcommand{\curle}{\preccurlyeq}
\renewcommand{\le}{\leqslant}
\renewcommand{\ge}{\geqslant}

\newenvironment{E6}[6]{%    modified nochmal  !!
{\footnotesize\begin{tabular}{@{}c@{}}
{#1}{#2}\lower2.3ex\vbox{\hbox{{#3}\rule{0ex}{0.1ex}}
\hbox{\hspace{0.3ex}\rule{0ex}{.1ex}\rule{0ex}{.1ex}}\hbox{{#6}\strut}}{#4}{#5}
\end{tabular}}}

\begin{document}
\setlength{\parskip}{2pt plus 4pt minus 0pt}
\hfill {\scriptsize September 26, 2018} 
\vskip1.5ex

\title[Glorious pairs and Abelian ideals]%
{Glorious pairs of roots and Abelian ideals of a Borel subalgebra}
\author{Dmitri I. Panyushev}
\address[]{Institute for Information Transmission Problems of the R.A.S., Bolshoi Karetnyi per. 19,  
127051 Moscow,  Russia}
\email{panyushev@iitp.ru}
\thanks{This research is partially supported by the R.F.B.R. grant {\rus N0} 16-01-00818.}

\begin{abstract}
Let $\g$ be a simple Lie algebra with a Borel subalgebra $\be$. Let $\Delta^+$ be the corresponding 
(po)set of positive roots and $\theta$ the highest root. A pair $\{\eta,\eta'\}\subset \Delta^+$ is said to be 
glorious, if $\eta,\eta'$ are incomparable and $\eta+\eta'=\theta$. Using the theory of abelian ideals of 
$\be$, we {\bf (1)} establish a relationship of $\eta,\eta'$ to certain abelian ideals associated with long simple roots, {\bf (2)} provide a natural bijection between the glorious pairs and the pairs of adjacent long 
simple roots (i.e., some edges of the Dynkin diagram), and {\bf (3)} point out a simple transform 
connecting two glorious pairs corresponding to the incident edges in the Dynkin digram.  
In types $\GR{DE}{}$, we prove that if $\{\eta,\eta'\}$ 
corresponds to the edge through the branching node of the Dynkin diagram, then the meet 
$\eta\wedge\eta'$ is the unique maximal non-commutative root. There is also an analogue of this property
for all other types except type $\GR{A}{}$. As an application, we describe the minimal non-abelian ideals
of $\be$.
\end{abstract}

\maketitle

\section*{Introduction}

\noindent
Let $\g$ be a complex simple Lie algebra, with a Cartan subalgebra $\te$ and a triangular decomposition 
$\g=\ut\oplus\te\oplus \ut^-$. Here $\be=\ut\oplus\te$ is a fixed Borel subalgebra.  
In this note, we present new combinatorial properties of root systems and abelian ideals of $\be$.
Our setting is always combinatorial, i.e., the abelian ideals of $\be$, which are sums of 
root spaces of $\ut$, are identified with the corresponding sets of positive roots.  

Let $\Delta$ be the root system of $(\g,\te)$ in the $\BR$-vector space $V=\te_\BR^*$, 
$\Delta^+$ the set of positive roots in $\Delta$ corresponding to $\ut$,  
$\Pi$ the set of simple roots in $\Delta^+$, and $\theta$  the {highest root} in  $\Delta^+$. 
We regard $\Delta^+$ as poset with the usual partial ordering `$\curge$'. 
An {\it ideal\/} of $(\Delta^+,\curge)$ is a subset $I\subset \Delta^+$ such that if
$\gamma\in I, \nu\in\Delta^+$, and $\nu+\gamma\in\Delta^+$, then $\nu+\gamma\in I$. The set of 
minimal elements of $I$ is denoted by $\min(I)$.  Write also $\max(\Delta^+\setminus I)$ for the set of 
maximal elements of $\Delta^+\setminus I$. An ideal $I$ is {\it abelian}, if
$\gamma'+\gamma''\not\in \Delta^+$ for all $\gamma',\gamma''\in I$.

Let $\AD$ (resp. $\Ab$) be the set of {\it all\/} (resp. {\it all abelian}) ideals of $\Delta^+$. Then
$\Ab\subset\AD$ and we think of both sets as posets with respect to inclusion. The ideal {\it generated by\/} 
$\gamma\in\Delta^+$ is $I\langle{\curge}\gamma\rangle=\{\nu\in\Delta^+\mid \nu\curge \gamma\}\in\AD$, 
and  $\gamma$ is said to be {\it commutative}, if $I\langle{\curge}\gamma\rangle\in\Ab$. Write 
$\Delta^+_{\sf com}$ for the set of all commutative roots. For any $\Delta$,
this subset is explicitly described in~\cite[Theorem\,4.4]{jac06}. Note that $\Delta^+_{\sf com}\in \AD$.

A pair $\{\eta,\eta'\}\subset \Delta^+$ is said to be {\it glorious}, if $\eta,\eta'$ are incomparable in the 
poset $(\Delta^+,\curge)$, i.e., $\{\eta,\eta'\}$ is an antichain, and $\eta+\eta'=\theta$. Our goal is to  explore properties of glorious pairs in 
connection with the theory of abelian ideals. We prove that 
\\ \indent
(1) \ these $\eta, \eta'$ are necessarily long and commutative;  
\\ \indent
(2) \ there is a natural bijection between the glorious pairs and the
pairs of adjacent long simple roots, i.e., certain edges of the Dynkin diagram, 
see Theorem~\ref{thm:main};
\\ \indent
(3) \ there is a simple transformation connecting two glorious pairs corresponding to a
triple of consecutive long simple roots, i.e., two incident edges, see Theorem~\ref{thm:sosed-pary}. 
\\ \indent
(4) \ if $I\in\AD$ and $\#I {<} h^*$, where $h^*$ is the {\it dual Coxeter number} of $\Delta$, then $I\in\Ab$, see Lemma~\ref{lm:h*}. %\\ \indent

Let $\Pi_l$ be the set of {\bf long} simple roots. (In the $\GR{ADE}{}$-case, all roots are assumed to be 
long and hence $\Pi_l=\Pi$.) Since $\Pi_l$ is connected in the Dynkin diagram, we have 
$\#\{\text{glorious pairs}\}=\#\Pi_l-1$, and the above bijection (2) can be understood as a bijection
with the edges of the sub-diagram $\Pi_l$. 
In the non-simply laced case, there is a unique pair of adjacent simple roots 
of different length. This gives rise to a similar pair of roots (the {\it semi-glorious pair}), which is discussed in Section~\ref{subs:semi-glor}. %We also prove that 
To state more properties of glorious pairs, we need more notation. 
\\ \indent
Let $\Abo$ be the set of nonempty abelian ideals and $\Delta^+_l$  the set of {\bf long} positive roots.  
In~\cite[Sect.\,2]{imrn}, we constructed a map $\tau: \Abo \to \Delta^+_l$, which is onto (see
Section~\ref{sect:1}). 
If  $\tau(I)=\mu$, then $\mu\in\Delta^+_l$ is called the {\it rootlet\/} of $I$, also denoted $\rt(I)$.  
Letting $\Ab_\mu=\tau^{-1}(\mu)$, we get a partition of $\Abo$ parameterised by
$\Delta^+_l$. Each $\Ab_\mu$ is  a subposet of $\Ab$ and, moreover, 
$\Ab_\mu$ has a unique minimal and unique maximal element (ideal)~\cite[Sect.\,3]{imrn}.  
These extreme elements of $\Ab_\mu$  
are denoted by $I(\mu)_{\sf min}$ and $I(\mu)_{\sf max}$. 
\\  \indent 
\textbullet \quad We show that the glorious pairs are related to the abelian ideals $I(\ap)_{\sf min} $ with 
$\ap\in\Pi_l$.  If a glorious pair $\{\eta,\eta'\}$ gives rise to adjacent $\ap,\ap'\in \Pi_l$, then 
$\eta\in\min(I(\ap)_{\sf min})$, $\eta'\in\min(I(\ap')_{\sf min})$ and also $\kl(\eta)=\ap'$ and $\kl(\eta')=\ap$, see Prop.~\ref{pr:sosed-prostye}. 
Here $\kl({\cdot})$ is a natural map from $\Delta^+_{\sf com}$ to $\HP$, the set of affine simple roots, 
see Section~\ref{sect:1} for details. 
\\  \indent 
\textbullet \quad Conversely, given adjacent $\ap,\ap'\in \Pi_l$, the corresponding glorious pair
is the unique pair in $I(\ap)_{\sf min}\cup I(\ap')_{\sf min}$ that sums to $\theta$.
The constructions of $\{\eta,\eta'\}$ from $\{\ap,\ap'\}$ and vice versa heavily relies on properties of $I(\ap)_{\sf min} $, $I(\ap)_{\sf max}$, and the minuscule elements of the affine Weyl group $\HW$, as it was developed in \cite{imrn,jems}.
\\  \indent 
\textbullet \quad 
It is noticed in \cite[Sect.\,4]{jac06} and conceptually proved in \cite[Sect.\,3]{airs2} that if $\Delta$ is 
not of type $\GR{A}{n}$, then $\Delta^+_{\sf nc}:=\Delta^+\setminus \Delta^+_{\sf com}$ has a unique 
maximal element, denoted $\thi$. Then $\tthe:=\theta-\thi\in\Delta^+$ and $\thi\curle \tthe$.
In Section~\ref{sect:interval}, we prove that if $\Delta\in \{\GR{D-E}{}\}$, then the interval 
$[\thi,\tthe]$ in $\Delta^+$, which is a cube $\mathbb B^3$~\cite[Sect.\,5]{airs2}, is related to the glorious 
pairs associated with the edges through the branching node, $\bap$, of the Dynkin diagram. This is   deduced from some properties of the shortest element of $W$ that takes $\theta$ to $\bap$, see Theorem~\ref{thm:4.3}.
If $\Delta\in \{\GR{B-C-F-G}{}\}$, then
$\tthe=\thi\in \Pi$ and $\{\thi, \tthe\}$ is just the unique semi-glorious pair associated with the adjacent 
simple roots of different length.
\\  \indent 
\textbullet \quad 
In Section~\ref{sect:applic}, we prove that all minimal non-abelian ideals of $\be$ are related to either 
the glorious or semi-glorious pairs.
All these ideals have  the same cardinality, $h^*$. 

We refer to \cite{bour}, \cite[\S\,3.1]{t41} for standard results on root systems and Weyl groups. The 
numbering of simple roots for the exceptional Lie algebras follows~\cite[Tables]{t41}.

%%%%%%%%%%%%%%%%   
\section{Preliminaries}    \label{sect:1}

\noindent
Our main objects are $\Pi=\{\ap_1,\dots,\ap_n\}$, the vector space $V=\oplus_{i=1}^n{\mathbb R}\ap_i$, 
the  Weyl group $W$ generated by  simple reflections $s_\ap$ ($\ap\in\Pi$), and a $W$-invariant inner 
product $(\ ,\ )$ on $V$. Set $\rho=\frac{1}{2}\sum_{\nu\in\Delta^+}\nu$. The partial ordering `$\curle$' in $\Delta^+$ is defined by the rule  that $\mu\curle\nu$ if 
$\nu-\mu$ is a non-negative integral linear combination of simple roots. If 
$\mu=\sum_{i=1}^n c_i\ap_i\in\Delta^+$, then $\hot(\mu):=\sum_{i=1}^n c_i$ and $\{\mu:\ap_i\}:=c_i$. 
%\\ \indent
\subsection{}   \label{subs:1-1}
The theory of abelian ideals of $\be$ is built on their relationship, due to D.~Peterson, with 
the {\it minuscule elements\/} of the affine Weyl group $\HW$, see Kostant's account in~\cite{ko98}; 
another approach is presented in \cite{cp1}. Recall the necessary setup.
Letting $\widehat V=V\oplus {\mathbb R}\delta\oplus {\mathbb R}\lb$, we 
extend the inner product $(\ ,\ )$ on $\widehat V$ so that 
$(\delta,V)=(\lb,V)=(\delta,\delta)= (\lb,\lb)=0$ and $(\delta,\lb)=1$. Set  $\ap_0=\delta-\theta$, where
$\theta$ is the highest root in $\Delta^+$. Then

\begin{itemize}
\item[] \ 
$\widehat\Delta=\{\Delta+k\delta \mid k\in {\mathbb Z}\}$ is the set of affine
(real) roots; 
\item[] \ $\HD^+= \Delta^+ \cup \{ \Delta +k\delta \mid k\ge 1\}$ is
the set of positive affine roots; 
\item[] \ $\HP=\Pi\cup\{\ap_0\}$ is the corresponding set
of affine simple roots;
\item[] \  $\mu^\vee=2\mu/(\mu,\mu)$ is the coroot corresponding to 
$\mu\in \widehat\Delta$;
\end{itemize}
\noindent
For each $\ap\in \HP$, let $s_\ap$ denote the corresponding reflection in $GL(\HV)$.That is, 
$s_\ap(x)=x- (x,\ap)\ap^\vee$ for any $x\in \HV$. The affine Weyl group, $\HW$, is the subgroup of 
$GL(\HV)$ generated by the reflections $s_\ap$ ($\ap\in\HP$). 
The extended inner product $(\ ,\ )$ on $\widehat V$ is $\widehat W$-invariant. 
The {\it inversion set\/} of $w\in\HW$ is $\eus N(w)=\{\nu\in\HD^+\mid w(\nu)\in -\HD^+\}$.

\subsection{}    \label{subs:1-2}
Following D.\,Peterson, we say that $\hat w\in \HW$ is  {\it minuscule\/}, if 
$\eus N(\hat w)=\{\delta-\gamma \mid \gamma\in I_{\hat w}\}$ for some  $I_{\hat w}\subset \Delta$, see~\cite[2.2]{ko98}.
One then proves that {\sf (i)} $I_{\hat w}\subset \Delta^+$ and it is an abelian ideal,
{\sf (ii)} the assignment $\hat w\mapsto I_{\hat w}$ yields a bijection between the minuscule elements of
$\HW$ and the abelian ideals, see \cite{ko98},  \cite[Prop.\,2.8]{cp1}. 
Accordingly, if $I\in\Ab$, then $\hat w_I$ denotes the corresponding minuscule
element of $\HW$. Obviously, $\# I=\#\eus N(\hat w_I)=\ell(\hat w_I)$, where $\ell$ is the usual length function on $\HW$. 

\begin{prop}[cf.~{\cite[Thm.\,2.2 \& Thm.\,2.4]{imrn}}]    \label{prop:generators}
Let $J\in\Ab$.
\\ \indent
1) For $\gamma\in J$, we have $\gamma\in \min(J)$ if and only if 
$\hat w_J(\delta-\gamma)\in -\HP$;  
\\ \indent 2) if $\eta\not\in J$, then $J\cup \{\eta\} \in \Ab$ if and only if 
$\hat w_J(\delta-\gamma)=:\beta\in \HP$. If this is the case, then 
$\hat w_{J\cup\{\gamma\}}=s_\beta \hat w_J$.
\end{prop}

\subsection{}   \label{subs:1-3}
Using the above assertion, one defines a natural map $\kl: \Delta^+_{\sf com}\to \HP$ as follows. Given 
$\gamma\in\Delta^+_{\sf com}$, take any any $I\in\Ab$ such that $\gamma\in\min(I)$. (For instance,
$I=I\langle{\curge}\gamma\rangle$ will do.) Then 
$\hat w_I(-\delta+\gamma)=\beta \in \HP$ and we set $\kl(\gamma)=\beta$. The point is that this $\beta$ 
does not depend on the choice of $I$, see \cite[Theorem\,4.5]{jac06}. We say that $\kl(\gamma)$ is the
{\it class} of $\gamma$. It can be shown that $\kl(\gamma)\in\Pi$, whenever 
$\gamma\in\gH\setminus\{\theta\}$.
\subsection{}   \label{subs:1-4}
For any nontrivial minuscule $\hat w_I$, we have 
$\hat w_I(\ap_0)+\delta=\hat w_I(2\delta-\theta)\in \Delta^+_l$~\cite[Prop.\,2.5]{imrn}. This yields the 
mapping $\tau: \Abo\to \Delta^+_l$, which is onto~\cite[Theorem\,2.6(1)]{imrn}. That is, 
$\rt(I)=\hat w_I(2\delta-\theta)$. The {\it Heisenberg ideal}
\[
\gH:=\{\gamma\in \Delta^+ \mid (\gamma, \theta)\ne 0\}=
\{\gamma\in \Delta^+ \mid (\gamma, \theta)> 0\}\in\AD
\] 
plays a prominent role in the theory of 
abelian ideals and posets $\Ab_\mu=\tau^{-1}(\mu)$:

\textbullet\quad   $I=I(\mu)_{\sf min}$ for some $\mu\in\Delta^+_l$  if and only if 
$I\subset  \gH$~\cite[Theorem\,4.3]{imrn};

\textbullet\quad  $\# I(\mu)_{\sf min}=(\rho,\theta^\vee-\mu^\vee)+1$~\cite[Theorem\,4.2(4)]{imrn};

\textbullet\quad  For $I\in\Abo$, we have $I\in\Ab_\mu$ if and only if $I\cap\gH=I(\mu)_{\sf min}$~\cite[Prop.\,3.2]{jems};

If $I=I(\mu)_{\sf min}$, then $\hat w_I$ has the following simple description. First, there is a unique 
element of minimal length in $W$ that takes $\theta$ to $\mu$~\cite[Theorem\,4.1]{imrn}. Writing $w_\mu$ 
for this element, one has $\ell(w_\mu)=(\rho,\theta^\vee-\mu^\vee)$ and 
$\hat w_I=w_\mu s_{\ap_0}$~\cite[Theorem\,4.2]{imrn}.
\subsection{}   \label{subs:1-5}
Generalising Peterson's idea, Cellini--Papi associate a canonical  element of $\HW$ to 
{\bf any} $I\in\AD$, see \cite[Sect.\,2]{cp1}. Set $I^k:=I^{k-1}+I$ for $k\ge 2$, where 
$J+I=\{\nu+\eta\in\Delta^+\mid \nu\in J, \eta\in I\}$. If $m\in\BN$ is the smallest integer such that 
$I^m=\varnothing$, then $\hat w_I$ is the unique element of $\HW$ such that 
$\gN(\hat w_I)=\bigsqcup_{k=1}^{m-1} (k\delta-I^k)$. For the abelian ideals,
$I^2=\varnothing$ and one recovers the definition of minuscule elements.

Some basic results on abelian ideals of $\be$ that explore other directions of investigations can be 
found in \cite{cp3,cmp,ko04,jlt16,suter}.

%%%%%%%%%%%%%%%%%%%%%%%%%%%
\section{Some auxiliary results}   \label{sect:2}

\noindent
We begin with some properties of arbitrary ideals in $(\Delta^+,\curle)$. Recall that
$\#\gH=2h^*-3$, where $h^*:=(\rho, \theta^\vee)+1$ is the {\it dual Coxeter number\/} of $\Delta$.

\begin{lm}   \label{lm:sum-theta}
If $I\subset \Delta^+$ is a non-abelian ideal, then there are $\gamma_1,\gamma_2\in I$ such that
$\gamma_1+\gamma_2=\theta$. 
\end{lm}
\begin{proof}
If $\gamma_1+\gamma_2=\mu$ for some $\gamma_1,\gamma_2\in I$ and $\mu\ne \theta$, then there is
$\ap\in\Pi$ such that $\mu+\ap$ is a root. In this case, $\gamma_1+\ap$ or $\gamma_2+\ap$ is a root, 
necessarily in $I$, which provides an induction step.
\end{proof}

\begin{lm}     \label{lm:h*}
If $\hat I\in \AD$ is non-abelian, then $\# \hat I\ge h^*$. Conversely, if $I\subset \gH$ is an ideal and 
$\# I\ge h^*$, then $I$ is not abelian.
\end{lm}
\begin{proof}
1) If $\hat I$ is non-abelian, then so is $\hat I\cap\gH$ (Lemma~\ref{lm:sum-theta}). Therefore, we may assume that 
$\hat I\subset \gH$ and  $\gamma_1+\gamma_2=\theta$ for some $\gamma_1,\gamma_2\in \hat I$. Let 
$\hat w_{\hat I} \in \HW$ be the canonical element associated with $\hat I$, see Section~\ref{subs:1-5}. Since $\hat I^2=\{\theta\}$ and 
$\hat I^3=\varnothing$, we have
\[
    \gN(\hat w_{\hat I})=\{\delta-\gamma\mid \gamma\in \hat I\}\cup \{2\delta-\theta\} .
\]
Since $\HD^+\setminus  \gN(\hat w_{\hat I})$ is a closed under addition subset of $\HD^+$, $\gH\setminus \hat I$ does not contain pairs of summable (to $\theta$) elements. Hence $\gH\setminus \hat I$ is contained in a Lagrangian half of $\gH\setminus\{\theta\}$ and is strictly less (because $\hat I$ contains a summable pair $\gamma_1,\gamma_2$). Thus,
\[
   \#(\gH\setminus \hat I)=\#\gH - \# \hat I< (\#\gH-1)/2 . 
\]
Hence $\# \hat I> (\#\gH+1)/2=h^*-1$.
\\ \indent
2)  The set $\gH\setminus \{\theta\}$ consists of $h^*-2$ different pairs of roots that sums to $\theta$.
Therefore, if $\# I\ge h^*$, then $I$ contains at least one such summable pair.
\end{proof}

\begin{rema}  There are abelian ideals $I$ with $\#I \ge h^*$.
Such ideals are not contained in $\gH$. 
\end{rema}
\begin{lm}    \label{lm:I-&-gamma}
If $I\subset \gH$ is an abelian ideal and $I\cup\{\gamma\}$ is a non-abelian ideal
for some $\gamma\in\Delta^+$, then \ {\sf (i)} $\gamma\in\gH$,  \ 
{\sf (ii)} $I=I(\ap)_{\sf min}$ \ for some $\ap\in\Pi_l$, and \ {\sf (iii)} $\#I=h^*-1$.
\end{lm}
\begin{proof}  
Since $I$ is abelian, whereas $I\cup\{\gamma\}$ is not, there is $\mu\in I$ such that 
$\gamma+\mu=\theta$ (use Lemma~\ref{lm:sum-theta}). 
Hence $\gamma\in\gH$. By~\cite[Prop.\,4.2(4)]{imrn}, if $I$ is abelian, then $\#I \le (\rho,\theta^\vee)=h^*-1$. Furthermore,
$\#I=h^*-1$ if and only if $I=I(\ap)_{\sf min}$  for some $\ap\in\Pi$. Now, since $I\cup\{\gamma\}$ is 
non-abelian, we have $\#(I\cup\{\gamma\})\ge h^*$. Thus, $\#I=h^*-1$ and we are done.
\end{proof}

If $\eta+\eta'=\theta$, then we set $\tilde I=I\langle{\curge} \eta,\eta'\rangle=
\{\nu\in \Delta^+\mid \nu\curge \eta \text{ or } \nu\curge \eta'\}$. Then $\tilde I$ is not abelian and $\min(\tilde I)\subset \{\eta,\eta'\}$. Clearly, both roots $\eta,\eta'$ are either long or short.

\begin{lm}    \label{lm:I-hat}
If $\eta,\eta'\in\Delta^+_{\sf com}$ and $\eta+\eta'=\theta$, then 
\begin{itemize}
\item[\sf (i)] \ these two roots are incomparable and hence $\min(\tilde I)= \{\eta,\eta'\}$;
\item[\sf (ii)] \  the ideals $\tilde I\setminus \{\eta\}$ and $\tilde I\setminus \{\eta'\}$ are abelian;
\item[\sf (iii)] \ $\eta,\eta'\in \Delta_l$.
\end{itemize}
\end{lm}
\begin{proof}
{\sf (i)}  Obvious.

{\sf (ii)}  Assume that $\tilde I\setminus \{\eta\}$ is non-abelian. Then there are 
$\gamma_1,\gamma_2\in \tilde I\setminus \{\eta\}$ such that 
$\gamma_1+\gamma_2=\theta$. Since both $I\langle{\curge}\eta\rangle$  and $I\langle{\curge}\eta'\rangle$ 
are abelian, up to permutation of indices, one has
\begin{gather*}
\text{ $\gamma_1\in I\langle{\curge}\eta'\rangle$ and $\gamma_2\in I\langle{\curge}\eta\rangle\setminus\{\eta\}$,} \\ 
\text{ while  $\gamma_1\not\in I\langle{\curge}\eta\rangle$  and $\gamma_2\not\in  I\langle{\curge}\eta'\rangle$.}
\end{gather*}
This means that $\hot(\gamma_1)\ge \hot(\eta')$ and $\hot(\gamma_2)> \hot(\eta)$. Hence
$\hot(\gamma_1+\gamma_2)>\hot(\theta)$. A contradiction! Therefore, $\tilde I\setminus \{\eta\}\in\Ab$ and likewise for $\tilde I\setminus \{\eta'\}$.

{\sf (iii)}  Now, $J=\tilde I\setminus \{\eta,\eta'\}$ is an abelian ideal that has two abelian extensions:
$J \mapsto J\cup\{\eta\}$, $J \mapsto J\cup\{\eta'\}$. Let $\hat w_J\in \HW$ be the minuscule element
corresponding to $J$. The first extension means that 
$\hat w_J(\delta-\eta)=\beta$ for some $\beta\in\HP$. Likewise, $\hat w_J(\delta-\eta')=\beta'\in\HP$, 
see Prop.~\ref{prop:generators}.
Since $\eta+\eta'=\theta$, we have $\hat w_J(2\delta-\theta)=\hat w_J(2\delta-\eta-\eta')=\beta+\beta'$. 
On the other hand, $\hat w_J(2\delta-\theta)=\tau(J)$, and it is an element of $\Delta^+_l$, see 
Section~\ref{subs:1-4}. Hence $\beta,\beta'$ are adjacent simple roots of $\Delta$. Finally, because 
$\theta$ is long, then so is $\beta+\beta'$. This is only possible if $\beta,\beta'$ are long and hence 
$\eta,\eta'$ are also long.
\end{proof}

\begin{rema}
All the ideals $I(\ap)_{\sf min}$ ($\ap\in\Pi_l$) have the same cardinality, $h^*-1=(\rho,\theta^\vee)$; 
while this is not the case for $I(\ap)_{\sf max}$. A general formula for $\#(I(\ap)_{\sf max})$ is given in
\cite{cp3}.
\end{rema}

%%%%%%%%%%%% Section 3  %%%%%%%%%%%%
\section{Glorious pairs, adjacent long simple roots, and incident edges of the Dynkin diagram}   
\label{sect:3}

\noindent
\begin{df}  \label{def:glor-par}
A pair $\{\eta,\eta'\}\subset \Delta^+$ is said to be {\it glorious}, if these two roots are incomparable and 
$\eta+\eta'=\theta$.
\end{df}

Set $\thi=\sum_{\ap\in\Pi} \lfloor \{\theta:\ap\}/2\rfloor \ap$. If $\Delta$ is of
type $\GR{A}{n}$, then $\thi=0$. For all other types, one can prove uniformly that 
$\thi\in\Delta^+$,
$\thi$ is not commutative,  and $\thi$ is the unique maximal element of $\Delta^+_{\sf nc}:=\Delta^+\setminus \Delta^+_{\sf com}$, see~\cite[Section\,3]{airs2}.

\begin{lm}    \label{lm:icomparable}
Suppose that $\eta+\eta'=\theta$. Then $\eta,\eta'$ are incomparable $\Leftrightarrow$ they are
commutative.
\end{lm}
\begin{proof}
The implication `$\Leftarrow$' is obvious (cf. Lemma~\ref{lm:I-hat}(i)). To prove the other implication, assume
that $\eta\not\in\Delta^+_{\sf com}$. Then there are $\gamma_1,\gamma_2\curge \eta$ such that
$\gamma_1+\gamma_2=\theta$. Therefore, $\theta\curge 2\eta$. In other words, 
$\eta\curle \thi$. This clearly implies that $\eta'\curge \thi$ and hence 
$\eta'\curge\thi \curge \eta$.
\end{proof}

Note that  %{\sf \bfseries (i)} 
both elements of a glorious pair belong to $\gH$ and they are long in view of Lemma~\ref{lm:I-hat}(iii).

\subsection{From glorious pairs to adjacent simple roots}  \label{subs:v-odnu}
\leavevmode\par
\begin{prop}    \label{pr:sosed-prostye}
For a glorious pair $\{\eta,\eta'\}$ and $\tilde I=I\langle{\curge} \eta,\eta'\rangle$ as in 
Section~\ref{sect:2}, we have
\begin{itemize}
\item[\sf (i)] \  $\#\tilde I=h^*$ and $\tilde I=I(\ap)_{\sf min}\cup \{\eta'\}=I(\ap')_{\sf min}\cup \{\eta\}$ \ for 
some $\ap, \ap'\in\Pi_l$;
\item[\sf (ii)] \ the long simple roots $\ap, \ap'$ are adjacent in the Dynkin diagram;
\item[\sf (iii)] \ $\kl(\eta)=\ap'$ and $\kl(\eta')=\ap$;
\item[\sf (iv)] \ $\eta\in\min(I(\ap)_{\sf min})$ and $\eta'\in\min(I(\ap')_{\sf min})$.
\end{itemize}
\end{prop}
\begin{proof}
{\sf (i)} This follows from Lemmata~\ref{lm:I-&-gamma} and \ref{lm:I-hat}(ii).

{\sf (ii)} \  Using $J:=\tilde I\setminus\{\eta,\eta'\}$, the minuscule element $\hat w_J\in\HW$, and two
abelian extensions of $J$, as in the proof of Lemma~\ref{lm:I-hat}(iii), we obtain
\[
\text{ $\hat w_J(\delta-\eta)=\beta\in \Pi_l$, \ $\hat w_J(\delta-\eta')=\beta'\in\Pi_l$, \ 
and $\rt(J)=\beta+\beta'$.} 
\]
This also means that $\hat w_{J\cup\{\eta\}}=s_{\beta}\hat w_J$ and 
$\hat w_{J\cup\{\eta'\}}=s_{\beta'}\hat w_J$, see Prop.~\ref{prop:generators}. Then 
\begin{gather*}
\ap =\rt(I(\ap)_{\sf min})=\rt(J\cup\{\eta\})=s_{\beta}\hat w_J(2\delta-\theta)=s_\beta(\beta+\beta')=\beta'  , \\
\ap' =\rt(I(\ap')_{\sf min})=\rt(J\cup\{\eta'\})=s_{\beta'}\hat w_J(2\delta-\theta)=s_{\beta'}(\beta+\beta')=\beta .
\end{gather*}

{\sf (iii)} \  By the previous part of the proof, we have 
$\hat w_{J\cup\{\eta\}}=s_{\ap'}\hat w_J$ and $\hat w_{J\cup\{\eta'\}}=s_{\ap}\hat w_J$, and by 
definition of the {\it class} we have
\begin{gather*}
  \kl(\eta)=-s_{\ap'}\hat w_J(\delta-\eta)=-s_{\ap'}(\ap')=\ap', \\
  \kl(\eta')=-s_{\ap}\hat w_J(\delta-\eta')=-s_{\ap}(\ap)=\ap .   
\end{gather*}

{\sf (iv)} \ This follows from {\sf (i)}.
\end{proof}

\begin{rema}  Part (ii) of the proposition already appears in \cite[Prop.\,3.4(1)]{imrn}. However, we essentially reproduced a proof in order to have the necessary notation for Part (iii).
\end{rema}

It follows from Proposition~\ref{pr:sosed-prostye} that one associates an edge of the Dynkin (sub-)diagram
of $\Pi_l$ to any glorious pair. Below, we prove that this map is actually a bijection and obtain 
some more properties for the quadruple $(\eta,\eta',\ap,\ap')$ of Proposition~\ref{pr:sosed-prostye}.
To achieve this goal, one should be able to construct a glorious pair starting from a pair of adjacent long 
simple roots.

\subsection{From adjacent simple roots to glorious pairs}   
\label{subs:v-druguyu}
Let $\ap\in\Pi_l$. By \cite[Prop.\,4.6]{imrn}, there is a natural bijection between $\min(I(\ap)_{\sf min})$ 
and the simple roots adjacent to $\ap$, and this correspondence respects the length. More precisely, let 
$w_\ap\in W$ be the unique shortest element taking $\theta$ to $\ap$. If $\ap'\in\Pi$ is adjacent to $\ap$, 
then $\eta:=w_\ap^{-1}(\ap+\ap')\in \min(I(\ap)_{\sf min})$ corresponds to $\ap'$. In particular, 
$\eta\in\Delta^+_{\sf com}$. (Here $\ap'$ is not necessarily long and $\|\ap'\|=\|\eta\|$.) Suppose now 
that $\ap'$ is long. Then so is $\eta$. Furthermore,
$(\theta,\eta)=(\ap,\ap+\ap')>0$, hence $\eta':=\theta-\eta$ is a long root, too.

\begin{prop}   \label{prop:symm-construct}
If $\ap,\ap'\in\Pi_l$ are adjacent and $\eta=w_\ap^{-1}(\ap+\ap')$, then  
$\eta'=w_{\ap'}^{-1}(\ap+\ap')$ has the property that $\eta+\eta'=\theta$. Furthermore,
$\kl(\eta)=\ap'$ and $\kl(\eta')=\ap$.
\end{prop}
\begin{proof}
1) \ Recall that $\ell(w_{\ap+\ap'})=(\rho,\theta^\vee-(\ap+\ap')^\vee)=h^*-3$, see Section~\ref{subs:1-4}.
%~\cite[Theorem\,4.2]{imrn}. 
Then
$\ell(s_\ap w_{\ap+\ap'})\le h^*-2=(\rho,\theta^\vee-(\ap')^\vee)$ and $s_\ap w_{\ap+\ap'}$ takes $\theta$
to $\ap'$. Hence one must have $w_{\ap'}=s_{\ap}w_{\ap+\ap'}$ and also
$w_\ap=s_{\ap'}w_{\ap+\ap'}$. (The first equality means that a ``shortest reflection route" inside $\Delta^+_l$ from 
$\theta$ to $\ap'$ can be realised as a shortest route from $\theta$ 
to $\ap+\ap'$ and then the last step $\ap+\ap'\stackrel{s_\ap}{\longrightarrow} \ap'$.) Now
$w_\ap^{-1}(\ap+\ap')=w_{\ap+\ap'}^{-1}(\ap')$. Hence \\
\centerline{$w_\ap^{-1}(\ap+\ap')+w_{\ap'}^{-1}(\ap+\ap')=w_{\ap+\ap'}^{-1}(\ap'+\ap)=\theta$.}

2) To compute the class, we need the fact that the minuscule element of $I(\ap)_{\sf min}$ in 
$\HW$ is 
$w_\ap s_{\ap_0}$. (This element already appeared in Section~\ref{subs:v-odnu} as $s_{\ap'}\hat w_J$, which is the same, because $\hat w_J=w_{\ap+\ap'}s_{\ap_0}$ and  $s_{\ap'}w_{\ap+\ap'}=w_{\ap}$.) Therefore, by definition,
\beq   \label{eq:class-eta}
   \kl(\eta)=-w_\ap s_{\ap_0}(\delta-\eta)=-w_\ap(\theta-\eta)=-\ap+(\ap+\ap')=\ap' .
\eeq
And likewise for $\eta'$.
\end{proof}

\noindent
Thus, both $\eta$ and $\eta'$ are long, commutative, and sum to $\theta$, i.e., we have obtained a 
glorious pair associated with $(\ap,\ap')$. The above result shows that the construction of  $\eta$ and 
$\eta'$ is symmetric w.r.t. $\ap$ and $\ap'$, and the classes of $\eta$ and $\eta'$ are computable. 

\begin{rmk}  \label{rem:another-appr}
Proposition~\ref{prop:symm-construct} provides an explicit formula for $\eta,\eta'$ via adjacent 
$\ap,\ap'\in\Pi_l$. There is also another approach. By~\cite[Theorem\,2.1]{imrn}, for two adjacent long simple roots, 
$I(\ap)_{\sf min}\cap I(\ap')_{\sf min}=I(\ap+\ap')_{\sf min}$. Furthermore,
$\# I(\ap+\ap')_{\sf min}=(\rho, \theta^\vee-(\ap+\ap')^\vee)+1=h^*-2$, see Section~\ref{subs:1-4}. 
Therefore, 
$\#\bigl(I(\ap)_{\sf min}\setminus I(\ap+\ap')_{\sf min}\bigr)=1$ and likewise for $\ap'$. Hence
$I(\ap)_{\sf min}\setminus I(\ap+\ap')_{\sf min}=\{\eta\}$ and 
$I(\ap')_{\sf min}\setminus I(\ap+\ap')_{\sf min}=\{\eta'\}$. Since $\#\bigl(
I(\ap)_{\sf min}\cup I(\ap')_{\sf min}\bigr)=h^*$ and
$I(\ap)_{\sf min}\cup I(\ap')_{\sf min}\subset\gH$, this ideal cannot be abelian, i.e., $\{\eta,\eta'\}$ is a glorious pair.
\end{rmk}

\subsection{Bijections and examples}
It is easily seen that the constructions of Sections~\ref{subs:v-odnu} and \ref{subs:v-druguyu} are mutually inverse. This yields one of the main results of the article.

\begin{thm}     \label{thm:main}
There is a natural bijection
\begin{gather*}
\{ \text{the glorious pairs in } \ \Delta^+\}  \stackrel{1:1}{\longleftrightarrow} \{ \text{the pairs of adjacent roots in } \Pi_l\}, \\
(\eta,\eta')  \stackrel{1:1}{\longmapsto}  (\ap,\ap')
\end{gather*} 
such that $\eta\in \min(I(\ap)_{\sf min})$, $\eta'\in \min(I(\ap')_{\sf min})$,
$\kl(\eta)=\ap'$, and $\kl(\eta')=\ap$.
\end{thm}

\begin{cl}
The number of glorious pairs equals $\#\Pi_l-1$. 
\end{cl}
This provides an {\sl a priori\/} proof for Theorem\,4.8(ii) in \cite{mz}, where glorious pairs appear disguised as ``summable 2-antichains'' in $\Delta(1):=\gH\setminus\{\theta\}$.
\begin{cl}     \label{cl:unique}
Every {\bfseries ordered} glorious pair $(\eta,\eta')$ arises as an element of the set 
$\min(I(\ap)_{\sf min})\times \max(\Delta^+\setminus I(\ap)_{\sf max})$ for a unique $\ap\in\Pi_l$.
\end{cl}
\begin{proof}
Given a glorious pair $(\eta,\eta')\in \min(I(\ap)_{\sf min})\times \min(I(\ap')_{\sf min})$, one can apply 
\cite[Theorem\,4.7]{jems} in two different ways. For, one can say that either 
 $\eta'=\theta-\eta\in \max(\Delta^+\setminus I(\ap)_{\sf max})$ or 
 $\eta=\theta-\eta'\in \max(\Delta^+\setminus I(\ap')_{\sf max})$. Therefore,
 \begin{gather*}
 (\eta,\eta')\in  \min(I(\ap)_{\sf min})\times \max(\Delta^+\setminus I(\ap)_{\sf max}) , \\
(\eta',\eta)\in  \min(I(\ap')_{\sf min})\times \max(\Delta^+\setminus I(\ap')_{\sf max}) .   \qedhere
 \end{gather*}   
\end{proof}

\begin{ex}   \label{ex:Dn}
Let $\Delta$ be of type $\GR{D}{n}$, with $\ap_i=\esi_i-\esi_{i+1}$ ($1\le i\le n-1$) and
$\ap_n=\esi_{n-1}+\esi_n$.
Then $\theta=\esi_1+\esi_2$ and the quadruples $(\ap,\ap',\eta,\eta')$ are given in the table.

\begin{center}
\begin{tabular}{ >{$}c<{$} >{$}c<{$} |>{$}c<{$} >{$}c<{$} c }
\ap & \ap' & \eta & \eta' & \\ \hline
\ap_i & \ap_{i+1} & \esi_1-\esi_{i+2} & \esi_2+\esi_{i+2} & ($1{\le} i {\le} n{-}3)$ \\
\ap_{n-2} & \ap_{n-1} & \esi_1+\esi_{n} & \esi_2-\esi_{n} &  \\
\ap_{n-2} & \ap_{n} & \esi_1-\esi_{n} & \esi_2+\esi_{n} &  \\
\hline 
\end{tabular}
\end{center}

\noindent
Using these data, one easily determines $\min(I(\ap)_{\sf min})$ for all $\ap\in\Pi$. 
Given $\ap_i\in\Pi$, we look at all its occurrences in the first two columns an pick
the related roots $\eta$, if $\ap_i=\ap$,  and $\eta'$, if $\ap_i=\ap'$. This yields $\min(I(\ap_i)_{\sf min})$.
For instance, $\min(I(\ap_{n-2})_{\sf min})=\{\esi_2+\esi_{n-1}, \esi_1+\esi_{n},\esi_1-\esi_{n}\}$.
\end{ex}

\begin{ex}   \label{ex:Bn}
Let $\Delta$ be of type $\GR{B}{n}$, with $\ap_i=\esi_i-\esi_{i+1}$ ($1\le i\le n$), where $\esi_{n+1}=0$.
Then $\Pi_l=\{\ap_1,\dots,\ap_{n-1}\}$, $\theta=\esi_1+\esi_2$, and the quadruples $(\ap,\ap',\eta,\eta')$ are
\begin{center}
\begin{tabular}{ >{$}c<{$} >{$}c<{$} |>{$}c<{$} >{$}c<{$} c }
\ap & \ap' & \eta & \eta' & \\ \hline
\ap_i & \ap_{i+1} & \esi_1-\esi_{i+2} & \esi_2+\esi_{i+2} & ($1{\le} i {\le} n{-}2)$ \\
%\hline 
\end{tabular}
\end{center}
\end{ex}
\begin{ex}   \label{ex:E6}
In the following table, we illustrate the bijection of Theorem~\ref{thm:main} and Corollary~\ref{cl:unique} 
with data for $\Delta$ of type $\GR{E}{6}$. The second column contains 10 different roots, which form 
5 glorious pairs. For each root there, one has to pick the unique root such that their sum equals 
$\theta=\left(\text{\begin{E6}{1}{2}{3}{2}{1}{2}\end{E6}}\right)$. The same 10 roots occur in the third 
column, and reading the table along the rows, we meet all 10 ordered glorious pairs.

\begin{center}
\begin{tabular}{ >{$}c<{$}| c|c |}
  &  $\min(I(\ap_i)_{\sf min})$  & $\max(\Delta^+\setminus I(\ap_i)_{\sf max})$ \\ \hline \hline
\ap_1 &  {\begin{E6}{1}{1}{1}{0}{0}{1}\end{E6}}  & {\begin{E6}{0}{1}{2}{2}{1}{1}\end{E6}}  \\
\ap_2 & {\begin{E6}{1}{1}{1}{1}{0}{1}\end{E6}}\ ; {\begin{E6}{0}{1}{2}{2}{1}{1}\end{E6}} 
&  {\begin{E6}{0}{1}{2}{1}{1}{1}\end{E6}}\ ; {\begin{E6}{1}{1}{1}{0}{0}{1}\end{E6}} \\
\ap_3 & {\begin{E6}{0}{1}{2}{1}{1}{1}\end{E6}}\ ; {\begin{E6}{1}{1}{1}{1}{1}{1}\end{E6}}\ ;
{\begin{E6}{1}{1}{2}{1}{0}{1}\end{E6}} &  
{\begin{E6}{1}{1}{1}{1}{0}{1}\end{E6}}\ ; {\begin{E6}{0}{1}{2}{1}{0}{1}\end{E6}}\ ; 
{\begin{E6}{0}{1}{1}{1}{1}{1}\end{E6}} \\
\ap_4 & {\begin{E6}{0}{1}{1}{1}{1}{1}\end{E6}}\ ;  {\begin{E6}{1}{2}{2}{1}{0}{1}\end{E6}} & 
 {\begin{E6}{1}{1}{2}{1}{0}{1}\end{E6}}\ ; {\begin{E6}{0}{0}{1}{1}{1}{1}\end{E6}}  \\
\ap_5 & {\begin{E6}{0}{0}{1}{1}{1}{1}\end{E6}}  & {\begin{E6}{1}{2}{2}{1}{0}{1}\end{E6}}  \\
\ap_6 & {\begin{E6}{0}{1}{2}{1}{0}{1}\end{E6}} & {\begin{E6}{1}{1}{1}{1}{1}{1}\end{E6}} \\
\hline 
\end{tabular}
\end{center}
\end{ex}

\begin{rema}
There are no glorious pairs for $\GR{C}{n}$ and $\GR{G}{2}$, since $\#\Pi_l=1$ in these cases. 
\end{rema}

\subsection{Adjacent glorious pairs}
\label{subs:sosednie-pary}
Two glorious pairs are said to be {\it adjacent\/} if the corresponding edges are {\it incident\/} in the 
Dynkin diagram, i.e., have a common node (in $\Pi_l$). Here we prove that there is a simple transform 
connecting two adjacent glorious pairs. First, we need some notation.

Let $\ap_i,\ap_j,\ap_k$ be consecutive long simple roots in the Dynkin diagram, see figure

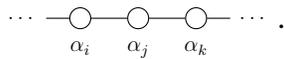
\centerline{
\raisebox{-2.5ex}{\begin{tikzpicture}[scale= .7, transform shape]
\draw (0,-0.2) node[below] {$\ap_i$} 
        (1.1,-0.2) node[below] {$\ap_j$} 
        (2.2,-0.2) node[below] {$\ap_k$};
%        (1.3,-1.1) node[right]  {$\delta_3$} ;
\tikzstyle{every node}=[circle, draw]
\node (a) at (0,0.1) {};
\node (b) at (1.1,0.1){};
\node (c) at (2.2,0.1) {};
%\node (d) at (1.1,-1.1) {};
\tikzstyle{every node}=[circle]
\node (g) at (-1.1,0.1)  {$\dots$};
\node (h) at (3.3,0.1)  {$\dots$};
\foreach \from/\to in {a/b, b/c, g/a, c/h}  \draw[-] (\from) -- (\to);
\end{tikzpicture}}.
}
It is not forbidden here that one of the nodes is the branching node (if there is any).

\noindent
To handle this situation, we slightly modify our notation.
The glorious pair associated with $(\ap_i,\ap_j)$ is denoted by $(\eta_{ij}, \eta_{ji})$ and likewise for 
$(\ap_j,\ap_k)$. Here $\eta_{ij}\in \min(I(\ap_i)_{\sf min})$, $\kl(\eta_{ij})=\ap_j$, etc. Set $w_i=w_{\ap_i}$, 
the shortest element of $W$ taking $\theta$ to $\ap_i$, and
likewise for the other admissible sets of indices; e.g.  $w_{jk}$ (resp. $w_{ijk}$) is the shortest element taking $\theta$ to $\ap_j+\ap_k$ (resp. $\ap_i+\ap_j+\ap_k$). Then
$\eta_{ij}=w_i^{-1}(\ap_i+\ap_j)$ and $\eta_{ji}=w_j^{-1}(\ap_i+\ap_j)$, cf. Section~\ref{subs:v-druguyu}.

\begin{thm}   \label{thm:sosed-pary}
If\/ $\ap_i,\ap_j,\ap_k$ are consecutive long simple roots, then
\[
    (\eta_{jk},\eta_{kj})=(\eta_{ij}-\gamma, \eta_{ji}+\gamma) ,
\] 
where $\gamma=w_{ijk}^{-1}(\ap_j)\in\Pi_l$ and $(\gamma,\theta)=0$.
\end{thm}
\begin{proof}
It is easily seen that  $w_i=s_js_kw_{ijk}, \ w_j=(s_is_k)w_{ijk}, \ w_k=s_js_i w_{ijk}$, cf. the proof of
Proposition~\ref{prop:symm-construct}(1).  Then
\begin{gather*}
\eta_{jk}=w_j^{-1}(\ap_j+\ap_k)=w_{ijk}^{-1}(s_ks_i)(\ap_j+\ap_k)=w_{ijk}^{-1}(\ap_i+\ap_j) ,
\\
\eta_{kj}=w_k^{-1}(\ap_j+\ap_k)=w_{ijk}^{-1}{\cdot}s_i s_j(\ap_j+\ap_k)=w_{ijk}^{-1}(\ap_k) .
\end{gather*}
Similarly,  $\eta_{ij}=w_{ijk}^{-1}(\ap_i)$ and $\eta_{ji}=w_{ijk}^{-1}(\ap_j+\ap_k)$. Therefore
\beq   \label{eq:beta1}
   \eta_{jk}-\eta_{ij}=w_{ijk}^{-1}(\ap_j)=\eta_{ji}-\eta_{kj} .
\eeq
Since $(\ap_i+\ap_j+\ap_k,\ap_j)=0$, it follows from \cite[Lemma\,2.1]{jlt16} that 
$w_{ijk}^{-1}(\ap_j)\in\Pi$ and $(w_{ijk}^{-1}(\ap_j),\theta)=0$.
\end{proof}

We will say that this $\gamma\in\Pi$ is the {\it transition root\/} for two incident edges $(\ap_i,\ap_j)$ and
$(\ap_j,\ap_k)$. If there is a longer chain of long simple roots, say 
\raisebox{-2.5ex}{\begin{tikzpicture}[scale= .7, transform shape]
\draw (0,-0.2) node[below] {$\ap_i$} 
        (1.1,-0.2) node[below] {$\ap_j$} 
        (2.2,-0.2) node[below] {$\ap_k$}
        (3.3,-0.2) node[below]  {$\ap_l$} ;
\tikzstyle{every node}=[circle, draw]
\node (a) at (0,0.1) {};
\node (b) at (1.1,0.1){};
\node (c) at (2.2,0.1) {};
\node (d) at (3.3,0.1) {};
\tikzstyle{every node}=[circle]
\node (g) at (-1.1,0.1)  {$\dots$};
\node (h) at (4.4,0.1)  {$\dots$};
\foreach \from/\to in {g/a, a/b, b/c, c/d, d/h}  \draw[-] (\from) -- (\to);
\end{tikzpicture}}, then one obtains two associated transition roots
$\gamma_j=w_{ijk}^{-1}(\ap_j)$ and $\gamma_k=w_{jkl}^{-1}(\ap_k)$. A useful complement to the previous theorem is 
\begin{prop}     \label{prop:4-roots}
The simple roots $\gamma_j, \gamma_k$ are adjacent in the Dynkin diagram, i.e., $(\gamma_j,\gamma_k)<0$,
and  $\eta_{kl}-\eta_{ij}=\gamma_j+\gamma_k=\eta_{ji}-\eta_{lk} \in\Delta^+$.
\end{prop}
\begin{proof}
Using the obvious notation, one has $w_{ijk}=s_l w_{ijkl}$ and  $w_{jkl}=s_i w_{ijkl}$. Therefore
$\gamma_j=w_{ijk}^{-1}(\ap_j)=w_{ijkl}^{-1}(\ap_j)$ and $\gamma_k=w_{jkl}^{-1}(\ap_k)=w_{ijkl}^{-1}(\ap_k)$.
Hence $(\gamma_j,\gamma_k)=(\ap_j,\ap_k)<0$. To prove the second relation, use 
Eq.~\eqref{eq:beta1} and its analogue for $(jkl)$.
\end{proof}

\begin{rmk}    \label{rmk:m-roots}
For a chain of $m$ consecutive {\bf long} simple roots, say $\ap_1,\dots,\ap_m$, one similarly obtains 
$m-2$ transition roots $\gamma_j$, $j=2,\dots,m-1$. Here $\gamma_2=w_{123}^{-1}(\ap_2)$, etc. It 
follows from Proposition~\ref{prop:4-roots} that $\gamma_i$ and $\gamma_{i+1}$ are adjacent in the 
Dynkin diagram. Moreover, arguing by induction on $m\ge 4$,
one easily proves that all $\gamma_i$'s are different, i.e., they form a genuine 
chain in $\Pi$. This provides a curious bijection in the $\GR{ADE}{}$-case. Namely, 
Theorem~\ref{thm:sosed-pary} asserts that each pair of incident edges in the Dynkin diagram defines 
a $\gamma\in\Pi_l$ such that $(\gamma,\theta)=0$. For $\GR{D}{n}$ and $\GR{E}{n}$, there is 
$n-1$ pair 
of incident edges in the Dynkin diagram and also $n-1$ simple roots orthogonal to $\theta$.
For $\GR{A}{n}$, the same phenomenon occurs with $n-2$ in place of $n-1$. Thus, one obtains
natural bijections related to the Dynkin diagram:
\[
        \left\{ \!\!\begin{array}{c} \text{triples of consecutive} \\ \text{simple roots}
        \end{array}\!\! \right\} 
        \leftrightarrow
    \{\text{pairs of incident edges}\}\leftrightarrow \{\gamma\in\Pi \mid (\gamma,\theta)=0\} . 
\]
That is, every simple root orthogonal to $\theta$ occurs as the transition root for a unique pair of incident edges. In~\ref{subs:A3}, we provide  this correspondence for $\GR{E}{6}$, $\GR{E}{7}$, $\GR{E}{8}$,
and $\GR{D}{8}$.
\\ \indent
This fails, however, in the non-simply laced case.
\end{rmk}

\section{Glorious pairs related to the interval $[\thi,\tthe]$, tails, and elements of $W$}
\label{sect:interval}

\noindent
Recall that $\thi=\sum_{\ap\in\Pi} \lfloor \{\theta:\ap\}/2\rfloor \ap$ and if $\Delta\ne \GR{A}{n}$, then
$\thi\in \Delta^+$. Outside type $\GR{A}{}$, $\theta$ is a multiple of a fundamental weight and there is a 
unique $\ap_\theta\in\Pi$ such that $(\theta,\ap_\theta)\ne 0$. Here $\{\theta:\ap_\theta\}=2$, hence 
$\{\thi:\ap_\theta\}=1$ and $\thi\in\gH$. Then $\tthe:=\theta-\thi$ also belongs to $\gH$ and 
$\thi\curle\tthe$. Set 
$\mathfrak J=\{\gamma\in\Delta^+\mid \thi\curle\gamma\curle\tthe\}$. By~\cite[Sect.\,5]{airs2}, if 
$\Delta\in \{\GR{D-E}{}\}$, then $\mathfrak J\simeq \mathbb B^3$ (boolean cube, see Figure~\ref{fig:interval});
and if $\Delta\in \{\GR{B-C-F-G}{}\}$, then $\mathfrak J=\{\thi,\tthe\}$.

Write $|M|$ for the sum of elements of a subset $M\subset\Pi$. Recall that $|M|\in\Delta^+$
if and only if $M$ is connected.
Our results below exploit the following fact for $\Delta$ that is not of type $\GR{A}{}$
(see~\cite[Prop.\,4.6]{airs2}):
\beq   \label{eq:bap}
\text{
there is a unique $\bap\in\Pi$ such that $|\Pi_l|+\bap\in \Delta^+$ and then
$w_{|\Pi_l|}^{-1}(\bap)=-\thi$.}
\eeq

Set $\eus O=\{\ap\in\Pi\mid \{\theta:\ap\} \text{ is odd}\}$. The roots in $\eus O$ are also said to be 
{\it odd}. It follows from the definition of $\thi$ that 
\beq     \label{eq:thety}
    \theta-2\thi= \tthe-\thi=|\eus O| . 
\eeq
\subsection{The glorious pairs related to $\mathfrak J$ in the $\{\GR{D-E}{}\}$-case}
In types $\GR{D}{}$ and $\GR{E}{}$, there are exactly three odd roots. They are denoted by 
$\beta_1,\beta_2, \beta_3$. In this case, $\bap\in \Pi$ of Eq.~\eqref{eq:bap} is the branching node in 
the Dynkin diagram and $\Pi\setminus\{\bap\}$ has three connected components, which are called 
{\it tails}. Each tail $\ct_i$ is a chain, and it is also regarded as a subset of  $\Pi$. Then 
$|\ct_i|\in \Delta^+$, $\Pi\setminus\{\bap\}=\sqcup_{i=1}^3 \ct_i$,  and 
\beq   \label{eq:union-xvosty}
     |\Pi|=\sum_{i=1}^3|\ct_i|+\bap.  
\eeq
Set $\widehat\ct_i=\ct_i\cup\{\bap\}$. Then $|\widehat\ct_i|\in\Delta^+$ and $\Pi\setminus \widehat\ct_i$ is 
a disconnected subset of the Dynkin diagram. Therefore, $|\Pi|-|\widehat\ct_i|\not\in \Delta^+$
and hence $(|\Pi|,|\widehat\ct_i|^\vee)\ge 0$. Since $\gN(w_\mu^{-1})=\{\gamma\in\Delta^+\mid (\gamma,\mu^\vee)=-1\}$  for any $\mu\in\Delta^+_l$~\cite[Theorem\,4.1]{imrn},
this means that $|\widehat\ct_i|\not\in \gN(w_{|\Pi|}^{-1})$, 
i.e., $w_{|\Pi|}^{-1}(|\widehat\ct_i|)\in\Delta^+$.
Moreover, the following is true.

\begin{thm}   \label{thm:xvosty}
After a suitable renumbering of the $\beta_i$'s, one has
\[
    w_{|\Pi|}^{-1}(|\widehat\ct_i|)=\beta_i \quad \text{for } \  i=1,2,3 .
\] 
\end{thm}
\begin{proof}
It follows from Eq.~\eqref{eq:union-xvosty} that
$ |\Pi|=\sum_{i=1}^3|\widehat\ct_i|-2\bap$. By~\eqref{eq:bap},  we also have
$w_{|\Pi|}^{-1}(\bap)=-\thi$. Therefore, 
\[
   \theta=w_{|\Pi|}^{-1}(|\Pi|)=\sum_{i=1}^3 w_{|\Pi|}^{-1}(|\widehat\ct_i|)+2\thi .
\]
Using Eq.~\eqref{eq:thety},  we obtain
$\beta_1+\beta_2+\beta_3=\sum_{i=1}^3 w_{|\Pi|}^{-1}(|\widehat\ct_i|)$, and each summand in the right-hand
side is a positive root. Hence the assertion!
\end{proof}
\begin{cl}    \label{cor:xvosty}
$w_{|\Pi|}^{-1}(|\ct_i|)=\thi+\beta_i$. 
\end{cl}
It is assumed below that numbering of the tails and odd roots is matched as prescribed by 
Theorem~\ref{thm:xvosty}. The explicit correspondence $\ct_i\longleftrightarrow\beta_i$ is pointed out in~\ref{subs:A2}.

\begin{thm}     \label{thm:4.3}
Let $\nu_i\in\ct_i$ be the unique root adjacent to $\bap$. Then
$w_{\bap}^{-1}(\nu_i)=-\thi-\beta_i$.
\end{thm}
\begin{proof}
For each tail $\ct_i$, we define the element $w(\ct_i)\in W$ as the ordered product of simple reflections
$s_\ap$, where the $\ap$'s are taken along the chain $\ct_i$ and $s_{\nu_i}$ is the initial factor on the left, see figure: \vskip-1ex
\begin{center}
\begin{figure}[htbp]
\begin{tikzpicture}[scale=0.95, transform shape]
\draw (0,0.2) node[above] {\small $\nu_1$} 
        (1.1,0.2) node[above] {\small $\bap$}
        (2.2,0.2) node[above] {\small $\nu_2$} 
        (1.3,-1.1) node[right]  {\small $\nu_3$} ;
\node  (P) at (-.6,-.4) {$\underbrace{\phantom{territory}}_{\ct_1} $};
\node  (Q) at (2.8,-.4) {$\underbrace{\phantom{territory}}_{\ct_2} $};
\node  (R) at (.5,-1.1) {\scriptsize $\ct_3$:};
\tikzstyle{every node}=[circle, draw]
\node (a) at (0,0) {};
\node (b) at (1.1,0){};
\node (c) at (2.2,0) {};
\node (d) at (1.1,-1.1) {};
\tikzstyle{every node}=[circle]
\node (g) at (-1.1,0)  {$\dots$};
\node (h) at (3.3,0)  {$\dots$};
\foreach \from/\to in {a/b, b/c,  b/d, g/a, c/h}  \draw[-] (\from) -- (\to);
\end{tikzpicture}
\caption{Tails}   \label{fig:tails}   
\end{figure}
\end{center}

\vskip-1ex\noindent
Note that $w(\ct_i)$ and $w(\ct_j)$ commute, $w(\ct_i)^{-1}(\nu_i)=-|\ct_i|$, and
$w(\ct_j)^{-1}(\nu_i)=\nu_i$ if $i\ne j$. We can write
$w_{\bap}=w(\ct_1)\,w(\ct_2)\,w(\ct_3)\,w_{|\Pi|}$. Indeed, $w(\ct_i)|\Pi|=|\Pi|-|\ct_i|$ and hence
the element in the RHS takes $\theta$ to $|\Pi|-\sum_{i=1}^3 |\ct_i|= \bap$. Moreover,
its length is at most $(\rho,\theta^\vee)-1=(\rho,\theta^\vee-\bap^\vee)$. Hence it must be the unique element of minimal length taking $\theta$ to  $\bap$.
Then 
\[
   w_{\bap}^{-1}(\nu_i)=w_{|\Pi|}^{-1} \prod_{j=1}^3w(\ct_j)^{-1}(\nu_i)=w_{|\Pi|}^{-1}w(\ct_i)^{-1}(\nu_i)=
   w_{|\Pi|}^{-1}(-|\ct_i|)=-\thi-\beta_i .  \qedhere
\]
\end{proof}

We assume below that $\{i,j,k\}=\{1,2,3\}$.

\begin{cl}     \label{cor:xvost-&-glorious}
For any $i\in\{1,2,3\}$, 
\begin{itemize}
\item[\sf (i)] \ $w_{\bap}^{-1}(\bap+\nu_i)=\tthe-\beta_i=\thi+\beta_j+\beta_k$; 
\item[\sf (ii)] \ $w_{\nu_i}^{-1}(\bap+\nu_i)=\thi+\beta_i$.
\end{itemize}
\end{cl}
\begin{proof}
{\sf (i)} Use the equality $w_{\bap}^{-1}(\bap)=\theta$.

{\sf (ii)} This follows from {\sf (i)}  and Proposition~\ref{prop:symm-construct} with $\ap=\bap$ and 
$\ap'=\nu_i$.
\end{proof}

\noindent
It follows that the six roots strictly between $\thi$ and $\tthe$ form the three glorious pairs corresponding 
to the edges of the Dynkin diagram through $\bap$. Namely, 
\\
\centerline{the pair $\{\tthe-\beta_i,\thi+\beta_i\}$ corresponds to the edge $\{\bap,\nu_i\}$.}

\begin{center}
\begin{figure}[htbp]
\begin{tikzpicture}[scale=1.25] %transform shape]
\draw (.7,0) node {\small $\thi$}; 
\draw (-2.5,1.5) node {\small $\thi+\beta_1$}; 
\draw (-.05,1.1) node {\small $\thi+\beta_2$}; 
\draw (2.4,1.5) node {\small $\thi+\beta_3$}; 

\draw (1, 3.9) node {\small {\color{darkblue}$\beta_1$}}; 
\draw (1.8, 2.25) node {\small {\color{darkblue}$\beta_2$}}; 
\draw (1, .6) node {\small {\color{darkblue}$\beta_3$}}; 

\draw (-2.5,3) node {\small $\tthe-\beta_3$}; 
\draw (-.05,3.35) node {\small $\tthe-\beta_2$}; 
\draw (2.4,3) node {\small $\tthe-\beta_1$}; 
\draw (.7,4.5) node {\small $\tthe$}; 

\tikzstyle{every node}=[circle, draw] %, fill=pink!80]
\node (a) at (0,0) {};
\node[fill=brown!80] (b) at (-1.5,1.5) {};
\node[fill=brown!80] (c) at (0,1.5) {};
\node[fill=brown!80] (d) at (1.5,1.5) {};
\node (e) at (-1.5,3) {};
\node (f) at (0,3) {};
\node (g) at (1.5,3) {};
\node[fill=brown!80] (h) at (0,4.5) {};
\foreach \from/\to in {a/b, a/c, a/d, b/e, b/f, c/e, c/g, d/f, d/g, e/h, f/h, g/h}  \draw[-] (\from) -- (\to);
\end{tikzpicture}
\caption{The interval between $\thi$ and $\tthe$ for $\GR{D}{n}$ and $\GR{E}{n}$}   \label{fig:interval}   
\end{figure}
\end{center}
By~\cite[Appendix]{jems}, the {\it join\/} `$\vee$' of two elements of the poset $(\Delta^+,\curge)$ always 
exists. In \cite[Sect.\,2]{airs2}, we proved that the {\it meet\/} `$\wedge$' of two elements of $(\Delta^+,\curge)$ exists if and only if their supports are not disjoint. More precisely, if $\supp(\eta):
=\{\ap\in\Pi\mid \{\eta:\ap\}\ne 0\}$ and $\supp(\eta)\cap\supp(\eta')\ne \varnothing$, then
\[
     \eta\wedge\eta'=\sum_{\ap\in\Pi} \min\{\{\eta:\ap\},\{\eta':\ap\}\}\ap .%=:\min\{\eta,\eta'\} .
\]
In this case, $\eta\vee\eta'=\sum_{\ap\in\Pi} \max\{\{\eta:\ap\},\{\eta':\ap\}\}\ap$. Hence 
%Therefore, if $\supp(\eta)\cap\supp(\eta')\ne \varnothing$, then
$\eta+\eta'=\eta\wedge\eta'+\eta\vee\eta'$. Therefore, for any glorious pair, we have
$\eta\wedge\eta'+\eta\vee\eta'=\theta$.

\begin{cl}   \label{cor:inf-glorious}
If $\{\eta,\eta'\}$ is the glorious pair corresponding to  $\{\bap,\nu_i\}$, then
$\eta\wedge\eta'=\thi$ and $\eta\vee\eta'=\tthe$. 
\end{cl}
\begin{proof}
$(\thi+\beta_i)\wedge (\thi+\beta_j+\beta_k)=\thi$ and 
$(\tthe-\beta_j-\beta_k)\vee  (\tthe-\beta_i)=\tthe$.
\end{proof}

\begin{rmk}      \label{rmk:4.6}
For any glorious pair $\{\eta,\eta'\}$, the meet $\eta\wedge\eta'$  exists and belongs to $\gH$ 
whenever $\Delta$ is not of type $\GR{A}{n}$, see \cite[Remark\,2.6]{airs2}. 
Since $\eta,\eta'\in I\langle{\curge} (\eta\wedge\eta')\rangle$, we have 
$\eta\wedge\eta'\in \Delta^+_{\sf nc}$ and hence $\eta\wedge\eta'\curle \thi$.  It is just proved that 
this upper bound is attained for the edges through the branching node $\bap\in\Pi$. Using 
Theorem~\ref{thm:sosed-pary}, Proposition~\ref{prop:4-roots}, and Remark~\ref{rmk:m-roots},
one can prove that, for all other edges, the strict inequality $\eta\wedge\eta'\prec \thi$ holds. More 
precisely, let us say that the edges through $\bap$ are {\it central}. For any other edge $\eus E$ of the 
subdiagram $\Pi_l$, we consider the natural distance, $d(\eus E)$, from $\eus E$ to the closest central
edge. (That is, $d(\eus E)=1$  if $\eus E$ is not central, but is incident to a central 
edge, etc.) Suppose that $\{\eta,\eta'\}$ corresponds to $\eus E$ and 
\[
    \eus E=\eus E_0, \eus E_1,\dots, \eus E_d=\eus C ,
\]
is the unique chain of edges connecting $\eus E$ with the closest central edge $\eus C$. 
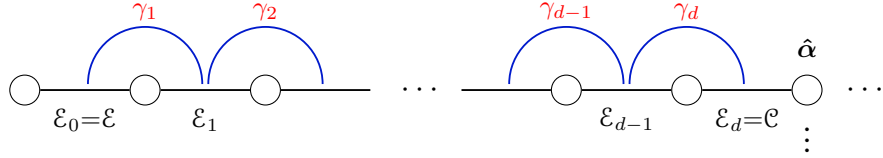
\begin{center}
\begin{tikzpicture}[scale= .8]

\node[circle,draw] (1) at (0,2) {};
\node[circle,draw] (2) at (2,2) {};
\node[circle,draw] (3) at (4,2) {};
\node[circle] (4) at (6,2) {};
\draw (6.6,2) node {$\cdots$};
\node[circle] (5) at (7,2) {};
\node[circle,draw] (6) at (9,2) {};
\node[circle,draw] (7) at (11,2) {};
\node[circle,draw] (8) at (13,2) {};
\foreach \from/\to in {1/2, 2/3, 3/4, 5/6, 6/7, 7/8} \draw [-,line width=.7pt] (\from) -- (\to);
\draw[darkblue,line width=.7pt] (2.95,2.1) arc(0:180:.95);
\draw[darkblue,line width=.7pt] (4.95,2.1) arc(0:180:.95);
\draw[darkblue,line width=.7pt] (9.95,2.1) arc(0:180:.95);
\draw[darkblue,line width=.7pt] (11.95,2.1) arc(0:180:.95);

\draw (2,3.3) node {{\color{red}\footnotesize $\gamma_1$}};
\draw (4,3.3) node {{\color{red}\footnotesize $\gamma_2$}};
\draw (9,3.3) node {{\color{red}\footnotesize $\gamma_{d-1}$}};
\draw (11,3.3) node {{\color{red}\footnotesize $\gamma_d$}};
\draw (1,1.5) node {\footnotesize $\eus E_0{=}\eus E$};
\draw (3,1.5) node {\footnotesize $\eus E_1$};
\draw (10,1.5) node {\footnotesize $\eus E_{d-1}$};
\draw (12,1.5) node {\footnotesize $\eus E_d{=}\eus C$};
\draw (13,2.7) node {\footnotesize $\bap$};
\draw (14,2) node {$\cdots$};
\draw (13,1.3) node {$\vdots$};

\end{tikzpicture}
\end{center}
Let $\gamma_i$ 
be the transition root for the incident edges $(\eus E_{i-1},\eus E_i)$, $i=1,\dots,d$, cf. Theorem~\ref{thm:sosed-pary}.
Then $d=d(\eus E)$ and $\eta\wedge\eta'=\thi-\sum_{i=1}^d \gamma_i$,
hence $\eta\vee\eta'=\tthe+\sum_{i=1}^d \gamma_i$. In particular, 
$\hot(\eta\wedge\eta')=\hot(\thi)-d(\eus E)$. 
\end{rmk}

\begin{rmk}      \label{rmk:4.7}
%The odd roots $\beta_1,\beta_2,\beta_3$ are the transition roots between the pairs of central edges.
If $\{i,j,k\}=\{1,2,3\}$, then the odd root $\beta_i$ is the transition root between the pair of central edges
$\{\bap,\nu_j\}$ and $\{\bap,\nu_k\}$.
\end{rmk}

All the roots but $\thi$ in {Figure}~\ref{fig:interval} are commutative, hence their classes are well-defined.

\begin{prop}    \label{prop:classes-interval}
We have $\kl(\thi+\beta_i)=\bap$, $\kl(\thi+\beta_j+\beta_k)=\nu_i$, and $\kl(\tthe)=\bap$.
\end{prop}
\begin{proof}
Since the roots between $\thi$ and $\tthe$ comprise three glorious pairs, the first two equalities follow from Proposition~\ref{prop:symm-construct} and  Corollary~\ref{cor:xvost-&-glorious}.
\\ \indent
The root $\tthe$ is adjacent to the roots $\tthe-\beta_i$ ($i=1,2,3$) in the Hasse diagram of $\Delta^+$, 
see Figure~\ref{fig:interval}.
Therefore, $\kl(\tthe)$ is adjacent to the $\kl(\tthe-\beta_i)=\nu_i$ ($i=1,2,3$) in the Dynkin diagram,
see~\cite[Lemma\,4.6]{jac06}. Hence the only possibility for $\kl(\tthe)$ is $\bap$.
\end{proof}

{\bf Remark.} The roots of class $\bap$ are coloured in Figure~\ref{fig:interval}.

\subsection{The interval $\mathfrak J$ in the $\{\GR{B-C-F-G}{}\}$-case}  
\label{subs:semi-glor}
The construction in Section~\ref{subs:v-druguyu} associates a glorious pair to any pair of adjacent long 
simple roots. This leaves nothing more to say in the simply laced case. But, in the non-simply laced
cases, there is exactly one pair of adjacent simple roots of different length, and one can construct a 
certain pair of roots as follows.

Let $\ap\in\Pi_l$, $\ap'\in\Pi_s:=\Pi\setminus\Pi_l$, and $(\ap,\ap')<0$. Then $\ap'=\bap$ and $\eta=w_\ap^{-1}(\ap+\ap')\in
\min(I(\ap)_{\sf min})$ is a {\bf short} root in $\gH$.  Hence $\eta\in\Delta^+_{\sf com}$ and 
$(\eta,\theta)=(\ap,\ap+\ap')>0$. But 
$\eta'=\theta-\eta\not\in\Delta^+_{\sf com}$ in view of Lemma~\ref{lm:I-hat}(iii). For this reason, we say 
that the pair $(\eta,\eta')$ is the (unique) {\it semi-glorious pair}. 

\begin{thm}   \label{thm:semi-glori}
Suppose that $\Delta$ is non-simply laced and $(\eta,\eta')$ is the semi-glorious pair associated with
the adjacent $\ap\in\Pi_l$ and $\ap'\in\Pi_s$, as above. Then 
\begin{itemize}
\item[\sf (i)] \  $I\langle{\curge}\eta'\rangle\setminus\{\eta'\}=I(\ap)_{\sf min}$ and $\eta-\eta'\in\Pi$;
\item[\sf (ii)] \  $\kl(\eta)=\ap'$;
\item[\sf (iii)] \ $(\eta,\eta')=(\tthe,\thi)$.
\end{itemize}
\end{thm}
\begin{proof}
{\sf (i)} \  %Assume that $\eta'\not\in\Delta^+_{\sf com}$, i.e., 
Since $I\langle{\curge}\eta'\rangle$ is not abelian, there are $\mu,\mu'\curge \eta$ such that $\mu+\mu'=\theta$ (Lemma~\ref{lm:sum-theta}). Because $\eta\in \min(I(\ap)_{\sf min})$, it follows from \cite[Theorem\,4.7]{jems} that 
$\eta'=\theta-\eta\in\max(\Delta^+\setminus I(\ap)_{\sf max})$. Hence 
$I\langle{\curge}\eta'\rangle\setminus\{\eta'\}$ lies in $I(\ap)_{\sf max}$ and thereby is abelian. Therefore, the only possibility is
$\mu=\eta$, $\mu'=\eta'$. In this case, we have $\eta'\prec \eta$.

Furthermore, since $I\langle{\curge}\eta'\rangle\setminus\{\eta'\}$ is abelian, while $I\langle{\curge}\eta'\rangle$ is not, Lemma~\ref{lm:I-&-gamma} implies that the former ideal is $I(\beta)_{\sf min}$ for some
$\beta\in \Pi_l$. By the construction, $I\langle{\curge}\eta'\rangle\setminus\{\eta'\} \subset I(\ap)_{\sf max}$.
Now, using the theory developed in \cite[Sect.\,3]{jems}, we obtain
\[
   I(\ap)_{\sf min}=\bigl( I(\ap)_{\sf max}\cap\gH\bigr) \supset 
   \bigl( I\langle{\curge}\eta'\rangle\setminus\{\eta'\}\bigr)=I(\beta)_{\sf min} ,
\]
which implies that $\beta=\ap$. Thus, $I\langle{\curge}\eta'\rangle\setminus\{\eta'\}=I(\ap)_{\sf min}$ contains
$\eta$ as a minimal element. Since all minimal elements of $I\langle{\curge}\eta'\rangle\setminus\{\eta'\}$ 
cover $\eta'$, we have $\eta-\eta'\in \Pi$.  

{\sf (ii)} \ The computation of Eq.~\eqref{eq:class-eta} applies here as well.

{\sf (iii)} \ Set $\gamma=|\Pi_l|$, the sum of all roots in $\Pi_l$. Then $\gamma\in\Delta^+$ and also
$\gamma+\ap'\in\Delta^+$. That is, $\ap'=\bap$ occurs in Eq.~\eqref{eq:bap}. Therefore,
$w_\gamma^{-1}(\ap')=-\thi$.
Let $\ct$ be the connected component of $\Pi\setminus \{\ap\}$ that contains long roots. Then $\ct$ is 
a chain, and if $\ct$ consists of consecutive (long) simple roots $\ap_{i_1}, \dots, \ap_{i_s}$, where 
$\ap_{i_s}$ is adjacent to $\ap$,  then we consider the element 
$\boldsymbol{w}=s_{i_s}\cdots s_{i_1}w_\gamma\in W$. Since $w_\gamma(\theta)=\gamma=
\ap_{i_1}+ \dots +\ap_{i_s}+\ap$, we have $\boldsymbol{w}(\theta)=\ap$. Because 
$\ell(\boldsymbol{w})\le \ell(w_\gamma)+ \#(\Pi_l)-1=(\rho,\theta^\vee)-1$, we see that 
$\ell(\boldsymbol{w})=(\rho,\theta^\vee)-1$ and $\boldsymbol{w}=w_\ap$.  Since 
$s_{i_j}(\ap')=\ap'$ for each $i_j$, we have
$
   w_\ap^{-1}(\ap')=w_\gamma^{-1}(\ap')=-\thi 
$.
Therefore, $\eta=w_\ap^{-1}(\ap+\ap')=\theta-\thi=\tthe$.
\end{proof}

The complete information on the semi-glorious pairs is presented in the following table:

\begin{center}
\begin{tabular}{ >{$}c<{$}| >{$}c<{$} >{$}c<{$} >{$}c<{$} >{$}c<{$} >{$}c<{$} |}
 & \ap & \ap' & \eta & \eta' & \eta-\eta' \\ \hline \hline 
\GR{B}{n} & \ap_{n-1} & \ap_n & \esi_1 & \esi_2 & \ap_1\in \Pi_l \\
\GR{C}{n} & \ap_{n} & \ap_{n-1} & \esi_1+\esi_n & \esi_1-\esi_n & \ap_n\in \Pi_l \\
\GR{F}{4} & \ap_{3} & \ap_{2} & (1221) & (1211) & \ap_3\in \Pi_l \\
\GR{G}{2} & \ap_{2} & \ap_{1} & (21) & (11) & \ap_1\in \Pi_s \\  \hline
\end{tabular}
\end{center}

\vskip1ex
\noindent
The numbering of simple roots follows \cite[Tables]{t41}, and the notation, say, (1221) stands for
$\ap_1+2\ap_2+2\ap_3+\ap_4$, and so on.

\begin{rmk}   \label{rem:1-odd-coeff}
Our result that $\tthe-\thi\in\Pi$ yields a proof for the fact that $\#\eus O=1$
in the $\GR{BCFG}{}$-case, cf. Eq.~\eqref{eq:thety}. 
\end{rmk}

There is an analogue of Theorem~\ref{thm:sosed-pary} that involves semi-glorious pairs. Let 
$\ap_i,\ap_j,\ap_k$ be consecutive simple roots such that $\ap_i,\ap_j\in\Pi_l$ and $\ap_k\in\Pi_s$, 
cf. Section~\ref{subs:sosednie-pary}. (This only happens for $\GR{B}{n}$ and $\GR{F}{4}$.) Then
we still have the shortest elements $w_i$ and $w_j$, but  
not $w_k$. Recall that $\eta_{ij}=w_i^{-1}(\ap_i+\ap_j)$, $\eta_{ji}=w_j^{-1}(\ap_i+\ap_j)$, and also
$\eta_{jk}=w_j^{-1}(\ap_j+\ap_k)$ is the commutative root in the semi-glorious pair. It follows from Theorem~\ref{thm:semi-glori}({\sf iii}) that
$\eta_{jk}=\tthe$, but this is not needed now.

\begin{thm}     \label{thm:semi-sosed}
 $\eta_{jk}-\eta_{ij}=:\gamma$ is a {\bfseries short} simple root and $(\gamma,\theta)=0$.
\end{thm}
\begin{proof}
In this situation, $\|\ap_j\|^2/\|\ap_k\|^2=2$ and therefore $\ap_i+\ap_j+2\ap_k\in \Delta^+_l$. Write
$w_{ij\un{k}}$ for the shortest element of $W$ taking $\theta$ to $\ap_i+\ap_j+2\ap_k$. Then
$w_i=s_j s_k w_{ij\un{k}}$ and $w_j=(s_is_k)w_{ij\un{k}}$. Therefore
\begin{gather*}
  \eta_{jk}=w_{ij\un{k}}^{-1}(s_ks_i)(\ap_j+\ap_k)=w_{ij\un{k}}^{-1}(\ap_i+\ap_j+\ap_k), 
  \\
  \eta_{ij}=w_{ij\un{k}}^{-1}s_k s_j (\ap_i+\ap_j)= w_{ij\un{k}}^{-1}(\ap_i) .
\end{gather*}
Hence $\eta_{jk}-\eta_{ij}=w_{ij\un{k}}^{-1}(\ap_j+\ap_k)=\gamma$ is a short root. Here $(\gamma,\theta)=
(\ap_i+\ap_j+2\ap_k, \ap_j+\ap_k)=0$, but we cannot blindly apply~\cite[Lemma\,2.1]{jlt16}, as in
%we obtain the same conclusion as in 
Theorem~~\ref{thm:sosed-pary}, because $\ap_j+\ap_k$ is not simple. Nevertheless, that proof can be 
adapted to this situation, see below.
\end{proof}

\begin{lm}
Under the above notation,  $\gamma:=w_{ij\un{k}}^{-1}(\ap_j+\ap_k)$ is a simple root.
\end{lm}
\begin{proof}
Set $\mu=\ap_i+\ap_j+2\ap_k$.
Since $\eus N(w_{\mu}^{-1})=\{\nu\in\Delta^+\mid (\nu,\mu^\vee)=-1\}$~\cite[Theorem\,4.1(2)]{imrn} 
and $(\mu,\ap_j+\ap_k)=0$, we have $\gamma\in\Delta^+$. Assume that 
$w_{\mu}^{-1}(\ap_j+\ap_k)=\gamma_1+\gamma_2$ is a sum of positive roots. Then 
$\ap_j+\ap_k=w_\mu(\gamma_1)+w_\mu(\gamma_2)$. 
\\ \indent
If both summands in the RHS are positive, then $\ap_j=w_\mu(\gamma_1)$, which is impossible, because $(\ap_j,\mu^\vee)=-1$ and hence $w_\mu^{-1}(\ap_j)$ is negative. Therefore,
without loss of generality, one may assume that $-\nu_1:=w_\mu(\gamma_1)$ is negative. Then 
$\nu_1\in \eus N(w_\mu^{-1})$, hence $(-\nu_1,\mu^\vee)=1$. Consequently,
$(\gamma_1,\theta^\vee)=1$. On the other hand, $0=(\mu,\ap_j+\ap_k)=(\theta,\gamma_1+\gamma_2)$ and therefore $(\theta, \gamma_2)<0$, which is impossible. Thus, $w_\mu^{-1}(\ap_j+\ap_k)$ must be simple. 
\end{proof}

\begin{rmk}   \label{rem:meet-for-BCFG}
If $(\eta,\eta')$ is the unique semi-glorious pair, then $\eta\vee\eta'=\eta=\tthe$; whereas for all 
genuine glorious pairs, one has $\eta\vee\eta' \succ\tthe$. Again, if a glorious pair $(\eta,\eta')$ 
corresponds to the edge $\eus E$ and $d(\eus E)$ is the distance from $\eus E$ to the edge of the semi-glorious pair, then $\eta\vee\eta'=\tthe+\sum_{i=1}^d \gamma_i$ and 
$\hot(\eta\vee\eta')=\hot(\tthe)+d(\eus E)$, cf. Remark~\ref{rmk:4.6}.
\end{rmk}

%%%%%%%%%%%%%%%%%%%%%%%%%
\section{The minimal non-abelian ideals of $\be$}
\label{sect:applic}

\noindent
As an application of previous theory, we characterise the minimal non-abelian ideals in
$\AD$ and point out the associated canonical elements in $\HW$ (cf.~Section~\ref{subs:1-5}). By 
Lemma~\ref{lm:h*}, if $\eus J\in \AD\setminus\Ab$ and $\#\eus J=h^*$, then
$\eus J$ is minimal non-abelian. The precise assertion is:

\begin{thm}     \label{thm:min-non-ab}
An ideal $\eus J\in\AD$ is minimal non-abelian if and only if either $\min(\eus J)$ is a glorious pair or
$\Delta$ is non-simply laced and $\eus J=I\langle{\curge}\thi\rangle$. For all these cases, 
$\eus J\subset\gH$, $\#\eus J=h^*$, and the unique summable (to $\theta$) pair of roots in $\eus J$ is either a glorious or semi-glorious pair.
\end{thm}
\begin{proof}
If $\eus J$ is non-abelian, then there are $\eta,\eta'\in\eus J$ such that $\eta+\eta'=\theta$. Hence
$I\langle {\curge}\eta,\eta'\rangle\subset \eus J$, and there must be equality if $\eus J$ is minimal.
There are two possibilities now.

{\it\bfseries (a)}  If $\eta,\eta'$ are incomparable, then $\{\eta,\eta'\}$ is a glorious pair and also 
$\#\eus J=h^*$, see Proposition~\ref{pr:sosed-prostye}(i). Here 
$\eus J=I(\ap)_{\sf min}\cup\{\eta'\}=I(\ap')_{\sf min}\cup\{\eta\}$. 

{\it\bfseries (b)}  Suppose that $\eta'\curle \eta$. Then $\eus J=I\langle {\curge}\eta'\rangle$ and also 
$\eta'\in\Delta^+_{\sf nc}$. (This already excludes the $\GR{A}{n}$-case, where all 1-generated ideals 
are abelian.) Then $\eta'\curle \thi$ and the minimality of $\eus J$ implies that $\eta'= \thi$.
However, in the $\GR{DE}{}$-case the non-abelian ideal $I\langle {\curge}\thi\rangle$ is not
minimal, since it properly contains three smaller non-abelian ideals generated by the glorious pairs related to the central edges of the Dynkin diagram, see Corollary~\ref{cor:xvost-&-glorious} and Figure~\ref{fig:interval}. In the $\GR{BCFG}{}$-case, we do obtain
a minimal non-abelian ideal of cardinality $h^*$, as follows from Theorem~\ref{thm:semi-glori}(i).
For, then $I\langle{\curge}\eta'\rangle\setminus\{\eta'\}=I(\ap)_{\sf min}$ and $\#I(\ap)_{\sf min}=h^*-1$.
\end{proof}  

A nice uniform formulation is that the minimal non-abelian ideals are in a one-to-one 
correspondence with the edges of the Dynkin diagram that are incident to the {\bf long} simple roots.

\begin{rmk} 
For a minimal non-abelian ideal $\eus J$, the canonical element $\hat w_{\eus J}\in \HW$ 
(see~\ref{subs:1-5}) can explicitly be described. Note that in both cases below, $\ell(\hat w_{\eus J})=\#\eus J+\#(\eus J^2)=h^*+1$.

{\it\bfseries (a)} \ Suppose that $\eus J=I\langle {\curge}\eta,\eta'\rangle$, where $(\eta,\eta')$ is the glorious pair 
corresponding to adjacent $\ap,\ap'\in \Pi_l$. Then
\[
   \hat w_{\eus J}=s_{\ap'} s_\ap w_{\ap}s_0=s_\ap s_{\ap'} w_{\ap'}s_0.
\]
For, $w_{\ap}s_0$ is the minuscule element for $I(\ap)_{\sf min}$ (Section~\ref{subs:1-4}). Hence 
$\gN(w_{\ap}s_0)=\{\delta-\gamma\mid \gamma\in I(\ap)_{\sf min}\}$. Then 
$\gN(s_\ap w_{\ap}s_0)=\gN(w_{\ap}s_0)\cup\{ (w_{\ap}s_0)^{-1}(\ap)\}  =\gN(w_{\ap}s_0)\cup \{2\delta-\theta\}$, 
and 
\[
  \gN(s_{\ap'}s_\ap w_{\ap}s_0)=\gN(s_\ap w_{\ap}s_0)\cup\{\delta-\eta'\}=
  \{\delta-\gamma\mid \gamma\in \eus J\}\cup \{2\delta-\theta\} .
\]
For the second equality, use the relations
$s_\ap w_{\ap+\ap'}=w_{\ap'}$, $s_{\ap'} w_{\ap+\ap'}=w_\ap$, and $(s_\ap s_{\ap'})^3=1$.

{\it\bfseries (b)} \ If $\eus J=I\langle {\curge}\thi\rangle$, then $\eus J$ can be thought of as the 
ideal related to the semi-glorious pair $(\tthe,\thi)$. If $\ap\in\Pi_l$ and $\ap'\in\Pi_s$ are the related 
adjacent roots, then similarly
\[
   \hat w_{\eus J}=s_{\ap'} s_\ap w_{\ap}s_0 .
\]
Note that here $w_{\ap'}$ does not exist and $(s_\ap s_{\ap'})^3\ne1$.
\end{rmk}

\appendix
\section{Data for the exceptional Lie algebras}
\label{sect:append}
\subsection{} \label{subs:A1}
For $\GR{E}{6}$, $\GR{E}{7}$, and $\GR{E}{8}$, the numbering of the simple roots is 
\\[.9ex] 
\raisebox{-3.1ex}
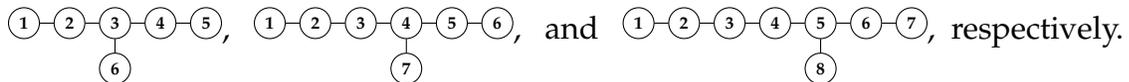
{\begin{tikzpicture}[scale= .55, transform shape]
\tikzstyle{every node}=[circle, draw]
\node (a) at (0,0) {\bf 1};
\node (b) at (1.1,0) {\bf 2};
\node (c) at (2.2,0) {\bf 3};
\node (d) at (3.3,0) {\bf 4};
\node (e) at (4.4,0) {\bf 5};
\node (f) at (2.2,-1.1) {\bf 6};
\foreach \from/\to in {a/b, b/c, c/d, d/e, c/f}  \draw[-] (\from) -- (\to);
\end{tikzpicture}}, \ \ 
\raisebox{-3.1ex}{\begin{tikzpicture}[scale= .55, transform shape]
\tikzstyle{every node}=[circle, draw] %, fill=brown!30]
\node (a) at (0,0) {\bf 1};
\node (b) at (1.1,0) {\bf 2};
\node (c) at (2.2,0) {\bf 3};
\node (d) at (3.3,0) {\bf 4};
\node (e) at (4.4,0) {\bf 5};
\node (f) at (5.5,0) {\bf 6};
\node (g) at (3.3,-1.1) {\bf 7};
\foreach \from/\to in {a/b, b/c, c/d, d/e, e/f, d/g}  \draw[-] (\from) -- (\to);
\end{tikzpicture}}, \ \
and \ \ 
\raisebox{-3.1ex}{\begin{tikzpicture}[scale= .55, transform shape]
\tikzstyle{every node}=[circle, draw] %, fill=white!55]
\node (h) at (-1.1,0) {\bf 1};
\node (a) at (0,0) {\bf 2};
\node (b) at (1.1,0) {\bf 3};
\node (c) at (2.2,0) {\bf 4};
\node (d) at (3.3,0) {\bf 5};
\node (e) at (4.4,0) {\bf 6};
\node (f) at (5.5,0) {\bf 7};
\node (g) at (3.3,-1.1) {\bf 8};
\foreach \from/\to in {h/a, a/b, b/c, c/d, d/e, e/f, d/g}  \draw[-] (\from) -- (\to);
\end{tikzpicture}},  \  respectively.  \\[1ex]
The glorious pairs $(\eta,\eta')$ associated with adjacent simple roots $(\ap,\ap')$ are given below.
Recall that here $\eta\in\min (I(\ap)_{\sf min})$ and $\kl(\eta)=\ap'$.
\vskip2ex

$\GR{E}{7}$: \quad 
\begin{tabular}{>{$}c<{$}>{$}c<{$} cc}
\ap & \ap' &  $\eta$ & $\eta'$ \\ \hline
\ap_1 & \ap_2 & 
\raisebox{-1.7ex}{\begin{tikzpicture}[scale= .7, transform shape]
\tikzstyle{every node}=[circle]
\node (a) at (0,0) {\bf 1};
\node (b) at (.3,0) {\bf 1};
\node (c) at (.6,0) {\bf 1};
\node (d) at (.9,0) {\bf 1};
\node (e) at (1.2,0) {\bf 1};
\node (f) at (1.5,0) {\bf 1};
\node (g) at (.9,-.5) {\bf 0};
\end{tikzpicture}}  
& 
\raisebox{-1.7ex}{\begin{tikzpicture}[scale= .7, transform shape]
\tikzstyle{every node}=[circle]
\node (a) at (0,0) {\bf 0};
\node (b) at (.3,0) {\bf 1};
\node (c) at (.6,0) {\bf 2};
\node (d) at (.9,0) {\bf 3};
\node (e) at (1.2,0) {\bf 2};
\node (f) at (1.5,0) {\bf 1};
\node (g) at (.9,-.5) {\bf 2};
\end{tikzpicture}}  
 \\    % 2-a stroka
 \ap_2 & \ap_3 & 
\raisebox{-1.7ex}{\begin{tikzpicture}[scale= .7, transform shape]
\tikzstyle{every node}=[circle]
\node (a) at (0,0) {\bf 1};
\node (b) at (.3,0) {\bf 1};
\node (c) at (.6,0) {\bf 1};
\node (d) at (.9,0) {\bf 1};
\node (e) at (1.2,0) {\bf 1};
\node (f) at (1.5,0) {\bf 1};
\node (g) at (.9,-.5) {\bf 1};
\end{tikzpicture}}  
& 
\raisebox{-1.7ex}{\begin{tikzpicture}[scale= .7, transform shape]
\tikzstyle{every node}=[circle]
\node (a) at (0,0) {\bf 0};
\node (b) at (.3,0) {\bf 1};
\node (c) at (.6,0) {\bf 2};
\node (d) at (.9,0) {\bf 3};
\node (e) at (1.2,0) {\bf 2};
\node (f) at (1.5,0) {\bf 1};
\node (g) at (.9,-.5) {\bf 1};
\end{tikzpicture}}  
 \\  % 3-a stroka
\ap_3 & \ap_4 & 
\raisebox{-1.7ex}{\begin{tikzpicture}[scale= .7, transform shape]
\tikzstyle{every node}=[circle]
\node (a) at (0,0) {\bf 1};
\node (b) at (.3,0) {\bf 1};
\node (c) at (.6,0) {\bf 1};
\node (d) at (.9,0) {\bf 2};
\node (e) at (1.2,0) {\bf 1};
\node (f) at (1.5,0) {\bf 1};
\node (g) at (.9,-.5) {\bf 1};
\end{tikzpicture}}  
& 
\raisebox{-1.7ex}{\begin{tikzpicture}[scale= .7, transform shape]
\tikzstyle{every node}=[circle]
\node (a) at (0,0) {\bf 0};
\node (b) at (.3,0) {\bf 1};
\node (c) at (.6,0) {\bf 2};
\node (d) at (.9,0) {\bf 2};
\node (e) at (1.2,0) {\bf 2};
\node (f) at (1.5,0) {\bf 1};
\node (g) at (.9,-.5) {\bf 1};
\end{tikzpicture}}  
 \\  % 4-a stroka
\ap_4 & \ap_5 & 
\raisebox{-1.7ex}{\begin{tikzpicture}[scale= .7, transform shape]
\tikzstyle{every node}=[circle]
\node (a) at (0,0) {\bf 1};
\node (b) at (.3,0) {\bf 1};
\node (c) at (.6,0) {\bf 2};
\node (d) at (.9,0) {\bf 2};
\node (e) at (1.2,0) {\bf 1};
\node (f) at (1.5,0) {\bf 1};
\node (g) at (.9,-.5) {\bf 1};
\end{tikzpicture}}  
& 
\raisebox{-1.7ex}{\begin{tikzpicture}[scale= .7, transform shape]
\tikzstyle{every node}=[circle]
\node (a) at (0,0) {\bf 0};
\node (b) at (.3,0) {\bf 1};
\node (c) at (.6,0) {\bf 1};
\node (d) at (.9,0) {\bf 2};
\node (e) at (1.2,0) {\bf 2};
\node (f) at (1.5,0) {\bf 1};
\node (g) at (.9,-.5) {\bf 1};
\end{tikzpicture}}  
 \\  % 5-a stroka
\ap_4 & \ap_7 & 
\raisebox{-1.7ex}{\begin{tikzpicture}[scale= .7, transform shape]
\tikzstyle{every node}=[circle]
\node (a) at (0,0) {\bf 1};
\node (b) at (.3,0) {\bf 1};
\node (c) at (.6,0) {\bf 1};
\node (d) at (.9,0) {\bf 2};
\node (e) at (1.2,0) {\bf 2};
\node (f) at (1.5,0) {\bf 1};
\node (g) at (.9,-.5) {\bf 1};
\end{tikzpicture}}  
& 
\raisebox{-1.7ex}{\begin{tikzpicture}[scale= .7, transform shape]
\tikzstyle{every node}=[circle]
\node (a) at (0,0) {\bf 0};
\node (b) at (.3,0) {\bf 1};
\node (c) at (.6,0) {\bf 2};
\node (d) at (.9,0) {\bf 2};
\node (e) at (1.2,0) {\bf 1};
\node (f) at (1.5,0) {\bf 1};
\node (g) at (.9,-.5) {\bf 1};
\end{tikzpicture}}  
 \\   % 6-a stroka
\ap_5 & \ap_6 & 
\raisebox{-1.7ex}{\begin{tikzpicture}[scale= .7, transform shape]
\tikzstyle{every node}=[circle]
\node (a) at (0,0) {\bf 1};
\node (b) at (.3,0) {\bf 2};
\node (c) at (.6,0) {\bf 2};
\node (d) at (.9,0) {\bf 2};
\node (e) at (1.2,0) {\bf 1};
\node (f) at (1.5,0) {\bf 1};
\node (g) at (.9,-.5) {\bf 1};
\end{tikzpicture}}  
& 
\raisebox{-1.7ex}{\begin{tikzpicture}[scale= .7, transform shape]
\tikzstyle{every node}=[circle]
\node (a) at (0,0) {\bf 0};
\node (b) at (.3,0) {\bf 0};
\node (c) at (.6,0) {\bf 1};
\node (d) at (.9,0) {\bf 2};
\node (e) at (1.2,0) {\bf 2};
\node (f) at (1.5,0) {\bf 1};
\node (g) at (.9,-.5) {\bf 1};
\end{tikzpicture}}  
 \\  \hline
\end{tabular}
%%%%%%%%%%%%%   å8  %%%%%%%%%%%%%%%%%%%%% 
%
\qquad
$\GR{E}{8}$: \quad 
\begin{tabular}{>{$}c<{$}>{$}c<{$} cc}
\ap & \ap' &  $\eta$ & $\eta'$ \\ \hline
\ap_1 & \ap_2 & 
\raisebox{-1.7ex}{\begin{tikzpicture}[scale= .7, transform shape]
\tikzstyle{every node}=[circle]
\node (h) at (-.3,0) {\bf 1};
\node (a) at (0,0) {\bf 2};
\node (b) at (.3,0) {\bf 2};
\node (c) at (.6,0) {\bf 2};
\node (d) at (.9,0) {\bf 2};
\node (e) at (1.2,0) {\bf 1};
\node (f) at (1.5,0) {\bf 0};
\node (g) at (.9,-.5) {\bf 1};
\end{tikzpicture}}  
&
\raisebox{-1.7ex}{\begin{tikzpicture}[scale= .7, transform shape]
\tikzstyle{every node}=[circle]
\node (h) at (-.3,0) {\bf 1};
\node (a) at (0,0) {\bf 1};
\node (b) at (.3,0) {\bf 2};
\node (c) at (.6,0) {\bf 3};
\node (d) at (.9,0) {\bf 4};
\node (e) at (1.2,0) {\bf 3};
\node (f) at (1.5,0) {\bf 2};
\node (g) at (.9,-.5) {\bf 2};
\end{tikzpicture}}  
\\  
\ap_2 & \ap_3 & 
\raisebox{-1.7ex}{\begin{tikzpicture}[scale= .7, transform shape]
\tikzstyle{every node}=[circle]
\node (h) at (-.3,0) {\bf 1};
\node (a) at (0,0) {\bf 2};
\node (b) at (.3,0) {\bf 2};
\node (c) at (.6,0) {\bf 2};
\node (d) at (.9,0) {\bf 2};
\node (e) at (1.2,0) {\bf 1};
\node (f) at (1.5,0) {\bf 1};
\node (g) at (.9,-.5) {\bf 1};
\end{tikzpicture}}  
&
\raisebox{-1.7ex}{\begin{tikzpicture}[scale= .7, transform shape]
\tikzstyle{every node}=[circle]
\node (h) at (-.3,0) {\bf 1};
\node (a) at (0,0) {\bf 1};
\node (b) at (.3,0) {\bf 2};
\node (c) at (.6,0) {\bf 3};
\node (d) at (.9,0) {\bf 4};
\node (e) at (1.2,0) {\bf 3};
\node (f) at (1.5,0) {\bf 1};
\node (g) at (.9,-.5) {\bf 2};
\end{tikzpicture}}  
\\
\ap_3 & \ap_4 & 
\raisebox{-1.7ex}{\begin{tikzpicture}[scale= .7, transform shape]
\tikzstyle{every node}=[circle]
\node (h) at (-.3,0) {\bf 1};
\node (a) at (0,0) {\bf 2};
\node (b) at (.3,0) {\bf 2};
\node (c) at (.6,0) {\bf 2};
\node (d) at (.9,0) {\bf 2};
\node (e) at (1.2,0) {\bf 2};
\node (f) at (1.5,0) {\bf 1};
\node (g) at (.9,-.5) {\bf 1};
\end{tikzpicture}}  
&
\raisebox{-1.7ex}{\begin{tikzpicture}[scale= .7, transform shape]
\tikzstyle{every node}=[circle]
\node (h) at (-.3,0) {\bf 1};
\node (a) at (0,0) {\bf 1};
\node (b) at (.3,0) {\bf 2};
\node (c) at (.6,0) {\bf 3};
\node (d) at (.9,0) {\bf 4};
\node (e) at (1.2,0) {\bf 2};
\node (f) at (1.5,0) {\bf 1};
\node (g) at (.9,-.5) {\bf 2};
\end{tikzpicture}}  
\\
\ap_4 & \ap_5 & 
\raisebox{-1.7ex}{\begin{tikzpicture}[scale= .7, transform shape]
\tikzstyle{every node}=[circle]
\node (h) at (-.3,0) {\bf 1};
\node (a) at (0,0) {\bf 2};
\node (b) at (.3,0) {\bf 2};
\node (c) at (.6,0) {\bf 2};
\node (d) at (.9,0) {\bf 3};
\node (e) at (1.2,0) {\bf 2};
\node (f) at (1.5,0) {\bf 1};
\node (g) at (.9,-.5) {\bf 1};
\end{tikzpicture}}  
&
\raisebox{-1.7ex}{\begin{tikzpicture}[scale= .7, transform shape]
\tikzstyle{every node}=[circle]
\node (h) at (-.3,0) {\bf 1};
\node (a) at (0,0) {\bf 1};
\node (b) at (.3,0) {\bf 2};
\node (c) at (.6,0) {\bf 3};
\node (d) at (.9,0) {\bf 3};
\node (e) at (1.2,0) {\bf 2};
\node (f) at (1.5,0) {\bf 1};
\node (g) at (.9,-.5) {\bf 2};
\end{tikzpicture}}  
\\
\ap_5 & \ap_6 & 
\raisebox{-1.7ex}{\begin{tikzpicture}[scale= .7, transform shape]
\tikzstyle{every node}=[circle]
\node (h) at (-.3,0) {\bf 1};
\node (a) at (0,0) {\bf 2};
\node (b) at (.3,0) {\bf 2};
\node (c) at (.6,0) {\bf 3};
\node (d) at (.9,0) {\bf 3};
\node (e) at (1.2,0) {\bf 2};
\node (f) at (1.5,0) {\bf 1};
\node (g) at (.9,-.5) {\bf 1};
\end{tikzpicture}}  
&
\raisebox{-1.7ex}{\begin{tikzpicture}[scale= .7, transform shape]
\tikzstyle{every node}=[circle]
\node (h) at (-.3,0) {\bf 1};
\node (a) at (0,0) {\bf 1};
\node (b) at (.3,0) {\bf 2};
\node (c) at (.6,0) {\bf 2};
\node (d) at (.9,0) {\bf 3};
\node (e) at (1.2,0) {\bf 2};
\node (f) at (1.5,0) {\bf 1};
\node (g) at (.9,-.5) {\bf 2};
\end{tikzpicture}}  
\\
\ap_5 & \ap_8 & 
\raisebox{-1.7ex}{\begin{tikzpicture}[scale= .7, transform shape]
\tikzstyle{every node}=[circle]
\node (h) at (-.3,0) {\bf 1};
\node (a) at (0,0) {\bf 2};
\node (b) at (.3,0) {\bf 2};
\node (c) at (.6,0) {\bf 2};
\node (d) at (.9,0) {\bf 3};
\node (e) at (1.2,0) {\bf 2};
\node (f) at (1.5,0) {\bf 1};
\node (g) at (.9,-.5) {\bf 2};
\end{tikzpicture}}  
&
\raisebox{-1.7ex}{\begin{tikzpicture}[scale= .7, transform shape]
\tikzstyle{every node}=[circle]
\node (h) at (-.3,0) {\bf 1};
\node (a) at (0,0) {\bf 1};
\node (b) at (.3,0) {\bf 2};
\node (c) at (.6,0) {\bf 3};
\node (d) at (.9,0) {\bf 3};
\node (e) at (1.2,0) {\bf 2};
\node (f) at (1.5,0) {\bf 1};
\node (g) at (.9,-.5) {\bf 1};
\end{tikzpicture}}  
\\
\ap_6 & \ap_7 & 
\raisebox{-1.7ex}{\begin{tikzpicture}[scale= .7, transform shape]
\tikzstyle{every node}=[circle]
\node (h) at (-.3,0) {\bf 1};
\node (a) at (0,0) {\bf 2};
\node (b) at (.3,0) {\bf 3};
\node (c) at (.6,0) {\bf 3};
\node (d) at (.9,0) {\bf 3};
\node (e) at (1.2,0) {\bf 2};
\node (f) at (1.5,0) {\bf 1};
\node (g) at (.9,-.5) {\bf 1};
\end{tikzpicture}}  
&
\raisebox{-1.7ex}{\begin{tikzpicture}[scale= .7, transform shape]
\tikzstyle{every node}=[circle]
\node (h) at (-.3,0) {\bf 1};
\node (a) at (0,0) {\bf 1};
\node (b) at (.3,0) {\bf 1};
\node (c) at (.6,0) {\bf 2};
\node (d) at (.9,0) {\bf 3};
\node (e) at (1.2,0) {\bf 2};
\node (f) at (1.5,0) {\bf 1};
\node (g) at (.9,-.5) {\bf 2};
\end{tikzpicture}}  
\\ \hline
\end{tabular}

\vskip2ex     %%%%%% tablitsa E6 %%%%%

$\GR{E}{6}$: \quad 
\begin{tabular}{>{$}c<{$}>{$}c<{$} cc}
\ap & \ap' &  $\eta$ & $\eta'$ \\ \hline
\ap_1 & \ap_2 & 
\raisebox{-1.7ex}{\begin{tikzpicture}[scale= .7, transform shape]
\tikzstyle{every node}=[circle]
\node (b) at (.3,0) {\bf 1};
\node (c) at (.6,0) {\bf 1};
\node (d) at (.9,0) {\bf 1};
\node (e) at (1.2,0) {\bf 0};
\node (f) at (1.5,0) {\bf 0};
\node (g) at (.9,-.5) {\bf 1};
\end{tikzpicture}}  
& 
\raisebox{-1.7ex}{\begin{tikzpicture}[scale= .7, transform shape]
\tikzstyle{every node}=[circle]
\node (b) at (.3,0) {\bf 0};
\node (c) at (.6,0) {\bf 1};
\node (d) at (.9,0) {\bf 2};
\node (e) at (1.2,0) {\bf 2};
\node (f) at (1.5,0) {\bf 1};
\node (g) at (.9,-.5) {\bf 1};
\end{tikzpicture}}  
 \\    % 2-a stroka
 \ap_2 & \ap_3 & 
\raisebox{-1.7ex}{\begin{tikzpicture}[scale= .7, transform shape]
\tikzstyle{every node}=[circle]
\node (b) at (.3,0) {\bf 1};
\node (c) at (.6,0) {\bf 1};
\node (d) at (.9,0) {\bf 1};
\node (e) at (1.2,0) {\bf 1};
\node (f) at (1.5,0) {\bf 0};
\node (g) at (.9,-.5) {\bf 1};
\end{tikzpicture}}  
& 
\raisebox{-1.7ex}{\begin{tikzpicture}[scale= .7, transform shape]
\tikzstyle{every node}=[circle]
\node (b) at (.3,0) {\bf 0};
\node (c) at (.6,0) {\bf 1};
\node (d) at (.9,0) {\bf 2};
\node (e) at (1.2,0) {\bf 1};
\node (f) at (1.5,0) {\bf 1};
\node (g) at (.9,-.5) {\bf 1};
\end{tikzpicture}}  
 \\    % 3-a stroka
\ap_3 & \ap_4 & 
\raisebox{-1.7ex}{\begin{tikzpicture}[scale= .7, transform shape]
\tikzstyle{every node}=[circle]
\node (b) at (.3,0) {\bf 1};
\node (c) at (.6,0) {\bf 1};
\node (d) at (.9,0) {\bf 2};
\node (e) at (1.2,0) {\bf 1};
\node (f) at (1.5,0) {\bf 0};
\node (g) at (.9,-.5) {\bf 1};
\end{tikzpicture}}  
& 
\raisebox{-1.7ex}{\begin{tikzpicture}[scale= .7, transform shape]
\tikzstyle{every node}=[circle]
\node (b) at (.3,0) {\bf 0};
\node (c) at (.6,0) {\bf 1};
\node (d) at (.9,0) {\bf 1};
\node (e) at (1.2,0) {\bf 1};
\node (f) at (1.5,0) {\bf 1};
\node (g) at (.9,-.5) {\bf 1};
\end{tikzpicture}}  
 \\    % 4-a stroka
\ap_3 & \ap_6 & 
\raisebox{-1.7ex}{\begin{tikzpicture}[scale= .7, transform shape]
\tikzstyle{every node}=[circle]
\node (b) at (.3,0) {\bf 1};
\node (c) at (.6,0) {\bf 1};
\node (d) at (.9,0) {\bf 1};
\node (e) at (1.2,0) {\bf 1};
\node (f) at (1.5,0) {\bf 1};
\node (g) at (.9,-.5) {\bf 1};
\end{tikzpicture}}  
& 
\raisebox{-1.7ex}{\begin{tikzpicture}[scale= .7, transform shape]
\tikzstyle{every node}=[circle]
\node (b) at (.3,0) {\bf 0};
\node (c) at (.6,0) {\bf 1};
\node (d) at (.9,0) {\bf 2};
\node (e) at (1.2,0) {\bf 1};
\node (f) at (1.5,0) {\bf 0};
\node (g) at (.9,-.5) {\bf 1};
\end{tikzpicture}}  
 \\    % 5-a stroka
\ap_4 & \ap_5 & 
\raisebox{-1.7ex}{\begin{tikzpicture}[scale= .7, transform shape]
\tikzstyle{every node}=[circle]
\node (b) at (.3,0) {\bf 1};
\node (c) at (.6,0) {\bf 2};
\node (d) at (.9,0) {\bf 2};
\node (e) at (1.2,0) {\bf 1};
\node (f) at (1.5,0) {\bf 0};
\node (g) at (.9,-.5) {\bf 1};
\end{tikzpicture}}  
& 
\raisebox{-1.7ex}{\begin{tikzpicture}[scale= .7, transform shape]
\tikzstyle{every node}=[circle]
\node (b) at (.3,0) {\bf 0};
\node (c) at (.6,0) {\bf 0};
\node (d) at (.9,0) {\bf 1};
\node (e) at (1.2,0) {\bf 1};
\node (f) at (1.5,0) {\bf 1};
\node (g) at (.9,-.5) {\bf 1};
\end{tikzpicture}}  
 \\    \hline
\end{tabular}
\qquad \quad $\GR{F}{4}$: \quad 
\begin{tabular}{>{$}c<{$}>{$}c<{$} cc}
\ap & \ap' &  $\eta$ & $\eta'$ \\ \hline
\ap_3 & \ap_4 & {\small 2211} & {\small 0221} \\
\end{tabular}

\subsection{} \label{subs:A2}
For $\GR{D}{n}$ and $\GR{E}{n}$, the explicit correspondence between tails and odd
roots is presented below. The numbering of tails follows Figure~\ref{fig:tails}. That is, according to the chosen numbering of $\Pi$,  see Example~\ref{ex:Dn} and figures in \ref{subs:A1}, $\ct_1$ contains $\ap_1$, $\ct_2$ contains $\ap_{n-1}$, and
$\ct_3$ contains $\ap_n$. In particular, we always have
$\#\ct_3=1$.
Then the odd root $\beta_i$ corresponding to $\ct_i$ is:

\begin{center}
\begin{tabular}{>{$}c<{$}| >{$}c<{$} >{$}c<{$} >{$}c<{$}| >{$}c<{$} |}
 & \beta_1 &\beta_2 & \beta_3 & \bap \\ \hline \hline
\GR{D}{n} & \ap_1 & \ap_{n-1} & \ap_n & \ap_{n-2} \\ 
\GR{E}{6} & \ap_1 & \ap_{5} & \ap_3 & \ap_{3} \\ 
\GR{E}{7} & \ap_1 & \ap_{5} & \ap_3 & \ap_{4} \\ 
\GR{E}{7} & \ap_2 & \ap_{8} & \ap_4 & \ap_{5} \\ \hline
\end{tabular}
\end{center}
\subsection{} \label{subs:A3}  Here we provide the transition roots for the pairs of incident edges  
in types $\GR{E}{6}$, $\GR{E}{7}$, and $\GR{E}{8}$, and thereby explicitly present the bijection of Remark~\ref{rmk:m-roots}.
\begin{center}
\begin{tikzpicture}[scale= .8]

\node (G) at (-2,1.5)  {$\GR{E}{6}$:};
\node[circle,draw] (2) at (2,2) {\tiny {\bf 1}};
\node[circle,draw] (3) at (4,2) {\tiny {\bf 2}};
\node[circle,draw] (4) at (6,2) {\tiny {\bf 3}};
\node[circle,draw] (5) at (8,2) {\tiny {\bf 4}};
\node[circle,draw] (6) at (10,2) {\tiny {\bf 5}};
\node[circle,draw] (7) at (6,0) {\tiny {\bf 6}};
\foreach \from/\to in {2/3, 3/4, 4/5, 5/6, 4/7} \draw [-,line width=.7pt] (\from) -- (\to);
%\draw[darkblue] (2.95,2.1) arc(0:180:.95);
\draw[darkblue,line width=.7pt] (4.95,2.1) arc(0:180:.95);
\draw[darkblue,line width=.7pt] (6.95,2.1) arc(0:180:.95);
\draw[darkblue,line width=.7pt] (8.95,2.1) arc(0:180:.95);
\draw[darkblue,line width=.7pt] (6.04,.95) arc(-90:0:.95);
\draw[darkblue,line width=.7pt] (5.04,1.9) arc(180:268:.95);
%\draw (2,3.3) node {{\color{red}\footnotesize $\ap_7$}};
\draw (4,3.3) node {{\color{red}\footnotesize $\ap_4$}};
\draw (6,3.3) node {{\color{red}\footnotesize $\ap_3$}};
\draw (8,3.3) node {{\color{red}\footnotesize $\ap_2$}};
\draw (5,1.1) node {{\color{red}\footnotesize $\ap_5$}};
\draw (7,1.1) node {{\color{red}\footnotesize $\ap_1$}};

\end{tikzpicture}
\end{center}

\vspace{.4cm}

\begin{center}
\begin{tikzpicture}[scale= .8]

\node (G) at (-2,1.5)  {$\GR{E}{7}$:};
\node[circle,draw] (1) at (0,2) {\tiny 1};
\node[circle,draw] (2) at (2,2) {\tiny 2};
\node[circle,draw] (3) at (4,2) {\tiny 3};
\node[circle,draw] (4) at (6,2) {\tiny 4};
\node[circle,draw] (5) at (8,2) {\tiny 5};
\node[circle,draw] (6) at (10,2) {\tiny 6};
\node[circle,draw] (7) at (6,0) {\tiny 7};
\foreach \from/\to in {1/2, 2/3, 3/4, 4/5, 5/6, 4/7} \draw [-,line width=.7pt] (\from) -- (\to);
\draw[darkblue,line width=.7pt] (2.95,2.1) arc(0:180:.95);
\draw[darkblue,line width=.7pt] (4.95,2.1) arc(0:180:.95);
\draw[darkblue,line width=.7pt] (6.95,2.1) arc(0:180:.95);
\draw[darkblue,line width=.7pt] (8.95,2.1) arc(0:180:.95);
\draw[darkblue,line width=.7pt] (6.04,.95) arc(-90:0:.95);
\draw[darkblue,line width=.7pt] (5.04,1.9) arc(180:268:.95);
\draw (2,3.3) node {{\color{red}\footnotesize $\ap_7$}};
\draw (4,3.3) node {{\color{red}\footnotesize $\ap_4$}};
\draw (6,3.3) node {{\color{red}\footnotesize $\ap_3$}};
\draw (8,3.3) node {{\color{red}\footnotesize $\ap_2$}};
\draw (5,1.1) node {{\color{red}\footnotesize $\ap_5$}};
\draw (7,1.1) node {{\color{red}\footnotesize $\ap_1$}};

\end{tikzpicture}
\end{center}

\vspace{.4cm}
%%%%%%%%  E8  %%%%%%%%%%

\begin{center}
\begin{tikzpicture}[scale= .8]

\node (G) at (-4,1.5)  {$\GR{E}{8}$:};
\node[circle,draw] (1) at (-2,2) {\tiny 1};
\node[circle,draw] (2) at (0,2) {\tiny 2};
\node[circle,draw] (3) at (2,2) {\tiny 3};
\node[circle,draw] (4) at (4,2) {\tiny 4};
\node[circle,draw] (5) at (6,2) {\tiny 5};
\node[circle,draw] (6) at (8,2) {\tiny 6};
\node[circle,draw] (7) at (10,2) {\tiny 7};
\node[circle,draw] (8) at (6,0) {\tiny 8};
\foreach \from/\to in {1/2, 2/3, 3/4, 4/5, 5/6, 6/7, 5/8} \draw [-,line width=.7pt] (\from) -- (\to);
\draw[darkblue,line width=.7pt] (0.95,2.1) arc(0:180:.95);
\draw[darkblue,line width=.7pt] (2.95,2.1) arc(0:180:.95);
\draw[darkblue,line width=.7pt] (4.95,2.1) arc(0:180:.95);
\draw[darkblue,line width=.7pt] (6.95,2.1) arc(0:180:.95);
\draw[darkblue,line width=.7pt] (8.95,2.1) arc(0:180:.95);
\draw[darkblue,line width=.7pt] (6.04,.95) arc(-90:0:.95);
\draw[darkblue,line width=.7pt] (5.04,1.9) arc(180:268:.95);
\draw (0,3.3) node {{\color{red}\footnotesize $\ap_7$}};
\draw (2,3.3) node {{\color{red}\footnotesize $\ap_6$}};
\draw (4,3.3) node {{\color{red}\footnotesize $\ap_5$}};
\draw (6,3.3) node {{\color{red}\footnotesize $\ap_4$}};
\draw (8,3.3) node {{\color{red}\footnotesize $\ap_3$}};
\draw (5,1.1) node {{\color{red}\footnotesize $\ap_8$}};
\draw (7,1.1) node {{\color{red}\footnotesize $\ap_2$}};
\end{tikzpicture}
\end{center}

\vspace{.5cm}
%%%%%%%%  D8  %%%%%%%%%%

\begin{center}
\begin{tikzpicture}[scale= .8]

\node (G) at (-4,1.5)  {$\GR{D}{8}$:};
\node[circle,draw] (1) at (-2,2) {\tiny 1};
\node[circle,draw] (2) at (0,2) {\tiny 2};
\node[circle,draw] (3) at (2,2) {\tiny 3};
\node[circle,draw] (4) at (4,2) {\tiny 4};
\node[circle,draw] (5) at (6,2) {\tiny 5};
\node[circle,draw] (6) at (8,2) {\tiny 6};
\node[circle,draw] (7) at (10,2) {\tiny 7};
\node[circle,draw] (8) at (8,0) {\tiny 8};
\foreach \from/\to in {1/2, 2/3, 3/4, 4/5, 5/6, 6/7, 6/8} \draw [-,line width=.7pt] (\from) -- (\to);
\draw[darkblue,line width=.7pt] (0.95,2.1) arc(0:180:.95);
\draw[darkblue,line width=.7pt] (2.95,2.1) arc(0:180:.95);
\draw[darkblue,line width=.7pt] (4.95,2.1) arc(0:180:.95);
\draw[darkblue,line width=.7pt] (6.95,2.1) arc(0:180:.95);
\draw[darkblue,line width=.7pt] (8.95,2.1) arc(0:180:.97);
\draw[darkblue,line width=.7pt] (8.04,.95) arc(-90:0:.95);
\draw[darkblue,line width=.7pt] (7.04,1.9) arc(180:268:.95);
\draw (0,3.3) node {{\color{red}\footnotesize $\ap_3$}};
\draw (2,3.3) node {{\color{red}\footnotesize $\ap_4$}};
\draw (4,3.3) node {{\color{red}\footnotesize $\ap_5$}};
\draw (6,3.3) node {{\color{red}\footnotesize $\ap_6$}};
\draw (8,3.3) node {{\color{red}\footnotesize $\ap_8$}};
\draw (7,1.1) node {{\color{red}\footnotesize $\ap_7$}};
\draw (9,1.1) node {{\color{red}\footnotesize $\ap_1$}};
\end{tikzpicture}
\end{center}

\end{document}